\documentclass[arxiv,reqno,twoside,a4paper,12pt]{amsart}

\usepackage{tikz}
\usepackage{amsmath, verbatim}
\usepackage{amssymb,amsfonts,mathrsfs}
\usepackage[colorlinks=true,linkcolor=blue,citecolor=blue]{hyperref}
\usepackage[driver=pdftex,margin=3cm,heightrounded=true,centering]{geometry}

\usepackage[scaled]{helvet} 
\usepackage{courier} 
\usepackage[mathbf]{euler}

\usepackage{mllocal} 

\newcommand{\scal}{\textup{scal}}

\newcommand{\U}{\mathscr{U}}

\newcommand{\xt}{\widetilde{x}}

\DeclareMathOperator{\ho}{\mathcal{C}^{\A}_{\textup{ie}}}

\DeclareMathOperator{\hok}{\mathcal{C}^{k, \A}_{\textup{ie}}}

\numberwithin{equation}{section}

\begin{document}

\title[Ricci de Turck flow on singular manifolds]
{Ricci de Turck flow on singular manifolds}

\author{Boris Vertman}
\address{Mathematisches Institut,
Universit\"at Oldenburg,
26129 Oldenburg,
Germany}
\email{boris.vertman@uni-oldenburg.de}

\date{\today}

\begin{abstract}
{In this paper we prove local existence of a Ricci de Turck flow starting at a space
with incomplete edge singularities and flowing for a short time 
within a class of incomplete edge manifolds. We derive 
regularity properties for the corresponding family of Riemannian metrics and 
discuss boundedness of the Ricci curvature along the flow. 
For Riemannian metrics that are sufficiently close to a flat incomplete edge metric,
we prove long time existence of the Ricci de Turck flow. Under certain conditions, 
our results yield existence of Ricci flow on spaces with incomplete edge singularities. The proof works 
by a careful analysis of the Lichnerowicz Laplacian and the Ricci de Turck flow equation.}
\end{abstract}

\maketitle
\tableofcontents

\section{Introduction and statement of the main result}

Geometric flows have attracted considerable interest and have been in the focus of extensive research in recent years,
among all most notably the Ricci flow which provided the decisive tool in the proof of Thurston's geometrization 
and the Poincare conjectures. In the present discussion we are interested in the Ricci flow $g(t)$ of an incomplete
manifold $(M,g_0)$ with an incomplete edge singular Riemannian metric $g_0$ satisfying the Ricci flow equation
\begin{equation}
\partial_t g(t) = -2 \, \textup{Ric} (g(t)), \quad g(0) = g_0.
\end{equation}
Such singular Ricci flows, which stay in a class of singular spaces, have been considered on 
K\"ahler manifolds in connection to a recent resolution of the Calabi-Yau conjecture for K\"ahler edge spaces 
by Jeffres, Mazzeo and Rubinstein \cite{JMR}. That paper arose in connection to the recent resolution of the 
Tian-Yau-Donaldson conjecture by Chen, Donaldson and Sun in \cite{Don1, Don2, Don3} and Tian \cite{Tian}. 
We also refer the reader to the survey by Rubinstein \cite{Rubin} on the background of the two conjectures.
In related very interesting developments, Chen and Wang \cite{Wang1}, Wang \cite{Wang2}, Liu and Zhang 
\cite{LZ} study existence and various properties of the conical K\"ahler Ricci flow. 
\medskip

In two dimensions, Ricci flow reduces to the Yamabe flow
and has been studied  by Mazzeo, Rubinstein and Sesum in \cite{MRS} and 
Yin \cite{Yin:RFO}. Yamabe flow of singular
edge manifolds in general dimension has been studied by the author in a joint work with Bahuaud in \cite{BV}. 
In the subsequent paper \cite{BV2} we study the long time behaviour of Yamabe flow of edge manifolds and solve
the Yamabe problem for incomplete edge metrics with a negative Yamabe invariant. Yamabe problem 
using elliptic methods has been studied by Akutagawa and Botvinnik in \cite{AkutagawaBotvinnik}
in case of isolated conical singularities, as well as by Akutagawa, Carron and Mazzeo in \cite{ACM} on edge manifolds.
\medskip

In the singular setting, Ricci flow need not be unique and 
alternatively to our treatment, Giesen-Topping \cite{Topping, Topping2}
obtained a solution to the Ricci flow on surfaces starting at a singular metric that becomes instantaneously complete. 
Moreover, Simon \cite{MS} studied Ricci flow in dimension two and three, where the singularity is smoothed out for positive times.  
\medskip

The setting of singular edge manifolds of dimension higher than two, which are not necessarily K\"ahler, is 
complicated since the Ricci flow equation does not reduce to a scalar equation and one 
is forced to study an equation of tensors. The present paper provides a first step into this direction and 
establishes short time existence of Ricci de Turck flow starting at and preserving a class of incomplete edge 
metrics. We point out that our analysis in particular applies to the setting
of isolated conical singularities. \medskip

We now proceed with an introduction into basic geometry of incomplete edge spaces,
definition of H\"older spaces on incomplete edge spaces,
outline the basic argument for short-time existence of the Ricci de Turck flow and 
formulation of the main results.

\subsection{Incomplete edge singularities}

\begin{defn}\label{admissible}
Consider an open interior $M$ of a compact manifold $\overline{M}$ with boundary $\partial M$. 
Let $\U = (0,1)_x \times \partial M$ be a tubular neighborhood
of the boundary in $M$ with the radial function $x:\U \to (0,1)$. 
Assume $\partial M$ is the total space of a fibration 
$\phi: \partial M \to B$ with the base $B$ and fibre $F$ being compact 
smooth manifolds, $\dim F \geq 1$. Consider a smooth Riemannian 
metric $g_B$ on the base manifold $B$ and a symmetric 
$2$-tensor $g_F$ on $\partial M$ which restricts to a fixed\footnote{In fact, 
the condition that $g_F$ restricts to a fixed metric on fibres is not necessary and is imposed here for simplicity. 
Either one only assumes that the metrics on the fibres are isospectral with respect to the tangential operator
of the Lichnerowicz Laplacian, or more generally one has to deal with a heat kernel that is only partially polyhomogeneous,
see Remark \ref{isospectrality} below.} Riemannian metric
on the fibres. We write $g_F$ for the Riemannian metrics on fibres as well. An incomplete edge 
metric $g$ on $M$ is defined here to be a smooth Riemannian 
metric such that $g=\overline{g}+h$ with $|h|_{\overline{g}}=O(x)$ and 

\begin{equation*}
\overline{g} \mid_{\, \U} = dx^2 + x^2 g_F + \phi^* g_B.
\end{equation*}
\end{defn}

The singular neighborhood $\U$ of such an incomplete edge space $M$
is illustrated in Figure \ref{figure1}. If $\dim B = 0$, the "edge"
reduces to a finite collection of isolated conical singularities.

\begin{figure}[h]
	\includegraphics[scale=0.75]{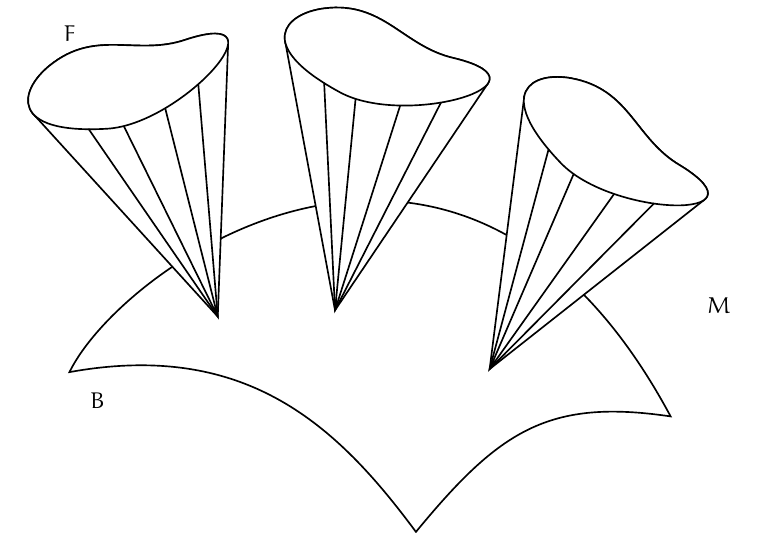}
	\caption{Incomplete Edge as a Cone bundle over $B$.}
	\label{figure1}
\end{figure}

We call such an edge metric \emph{admissible} if the fibration $\phi: (\partial M, g_F + \phi^*g_B) 
\to (B, g_B)$ is a Riemannian submersion. More precisely, we may split the tangent bundle $T_p\partial M$ 
into vertical and horizontal subspaces $T^V_p \partial M \oplus T^H_p \partial M$ as follows.
The vertical subspace $T^V_p\partial M$ is the tangent space to the fibre of $\phi$ through $p$, and the 
horizontal subspace $T^H_p \partial M$ is the annihilator 
of the subbundle $T^V_p\partial M \lrcorner g^F \subset T^*\partial M$ ($\lrcorner$ denotes contraction).  
Then $\phi$ is a Riemannian submersion if $g_F$ restricted to $T^H_p \partial M$ vanishes. 
Any level set $(\{x\}\times \partial M, x^2 g_F + \phi^*g_B)$ is then a Riemannian submersion as well.
\medskip

Other conditions on the metric $g$ will be added below and are related to the 
assumption of in a certain sense bounded curvature as well as the spectral analysis of the associated 
Laplace Beltrami and the Lichnerowicz Laplace operators.

\subsection{Geometry of incomplete edge spaces}\medskip

Choose local coordinates in the singular neighborhood $\U$ as follows. 
Consider local coordinates $(y)$ on $B$, lifted to $\partial M$ with respect to $\phi$, and then extended radially to $U$. 
Let coordinates $(z)$ restrict to local coordinates on fibres $F$. This defines
local coordinates $(x,y,z)$ in the neighborhood $\U$. \medskip

Consider the Lie algebra of edge vector fields $\mathcal{V}_e$, which by definition are smooth 
over $\overline{M}$ and at the boundary $\partial M$ tangent to the fibres of the fibration. 
In local coordinates, $\mathcal{V}_e$ is locally generated by (we write $b := \dim B$ and $f := \dim F$)
\[
\left\{x\frac{\partial}{\partial x}, x\partial_y = \left( x\frac{\partial}{\partial y_1}, \dots, x \frac{\partial}{\partial y_b}\right), 
\partial_z = \left( \frac{\partial}{\partial z_1},\dots, \frac{\partial}{\partial z_f} \right)\right\},
\]
with coefficients in the linear combinations of the derivatives being by definition smooth 
on $\overline{M}$. The vector bundle ${}^eTM$ over 
$\overline{M}$ is defined by requiring that the edge vector fields $\mathcal{V}_e$ 
form a spanning set of sections $\mathcal{V}_e=C^\infty(\overline{M},{}^eTM)$. The dual 
vector bundle of ${}^eTM$ is denoted by ${}^eT^*M$ and is generated locally by the following one-forms
\begin{align}\label{triv}
\left\{\frac{dx}{x}, \frac{dy_1}{x}, \dots, \frac{dy_b}{x}, dz_1,\dots,dz_f\right\}.
\end{align}
These differential one-forms, though singular in the usual sense, are smooth as sections of ${}^eT^*M$. 
We extend the radial function $x:\U \to (0,1)$ smoothly to $x \in C^\infty (\overline{M}, [0,\infty))$ such that $x^{-1}(\{0\})= \partial M$ and 
$dx \restriction \partial M \neq 0$. We define the vector bundle ${}^{ie}TM$ 
by asking\footnote{We write $x \, C^\infty(\overline{M},{}^{ie}TM) := 
\{x \cdot u \mid u \in C^\infty(\overline{M},{}^{ie}TM)\}$.} 
$x C^\infty(\overline{M},{}^{ie}TM) = C^\infty(\overline{M},{}^{e}TM)$. 
Its dual, the vector bundle ${}^{ie}T^*M$, is related to ${}^{e}T^*M$ by
\footnote{We write $x \, C^\infty(\overline{M},{}^{e}T^*M) 
:= \{x \cdot u \mid u \in C^\infty(\overline{M},{}^{e}T^*M)\}$.}  $C^\infty(\overline{M},{}^{ie}T^*M) = 
x C^\infty(\overline{M},{}^{e}T^*M)$, and is spanned locally by 
\begin{align}\label{triv2}
\left\{dx, dy_1, \dots, dy_b, x dz_1,\dots, x dz_f\right\}.
\end{align}
Construction of these vector bundles does not require a choice of a Riemannian metric on $M$.
Rather the vector bundles ${}^{ie}T^*M$ and ${}^{ie}T^*M$ allow us to express the structure of 
the complete edge metric $x^{-2} g$ and the incomplete edge metric $g$, as well as the corresponding curvatures in a convenient way. \medskip

The complete edge metric $x^{-2}g$ can be viewed as a smooth section of
the symmetric $2$-tensors on ${}^{e}T^*M$, which we write as $x^{-2}g \in \textup{Sym}^2({}^{e}T^*M)$. 
Therefore we refer to ${}^{e}TM$ and ${}^{e}T^*M$ as the complete tangent and cotangent bundles,
respectively. \medskip
 
The incomplete edge metric $g$ can be viewed as a smooth section of
the symmetric $2$-tensors on ${}^{ie}T^*M$, which we write as $g \in \textup{Sym}^2({}^{ie}T^*M)$. 
Therefore we refer to ${}^{ie}TM$ and ${}^{ie}T^*M$ as the incomplete tangent and cotangent bundles,
respectively. We adopt such a convention of incomplete Riemannian edge metrics viewed as 
sections of $\textup{Sym}^2({}^{ie}T^*M)$ 
from now whenever we don't say otherwise. Note also that the generators of ${}^{ie}TM$ and ${}^{ie}T^*M$
are of bounded length with respect to the Riemannian metric $g$ and its inverse, respectively. \medskip

The Riemannian curvature $(0,4)$ tensor $R(g)$ acting on 
$X_1,X_2, X_3, X_4 \in C^\infty(\overline{M},{}^{ie}TM)$ is generically 
$R(g) [X_1,X_2, X_3, X_4] \in x^{-2} C^\infty(\overline{M}) \equiv \{x^{-2} \cdot u \mid 
u \in C^\infty(\overline{M})\}$. We say in short that $R(g)$ acting on ${}^{ie}TM$ is generically
of order $O(x^{-2})$ as $x\to 0$. Similarly, the Ricci curvature tensor $\textup{Ric}(g)$ acting on ${}^{ie}TM$, 
as well as the scalar curvature $\scal(g)$, are generically of order $O(x^{-2})$ as $x\to 0$. However, there 
are geometrically interesting situations, where the Ricci curvature tensor on ${}^{ie}TM$
is bounded up to $x=0$. \medskip

First of all, there is of course the example of a flat cone over $\mathbb{S}^f$.
A second less trivial example is the case of a codimension two singularity, where 
the normal bundle $NB$ of $B$ inside $TM$ is a fibre bundle over $B$ with the fibre
being a two-dimensional disc $\mathbb{D}^2$. The involution on $\mathbb{D}^2$ defines a global
action $\sigma$ on the normal bundle $NB$, which may now be viewed as a branched covering of itself. 
Any $\sigma$-invariant smooth metric on $NB$ descends to a singular edge metric on $NB/_\sigma$
and extends smoothly to $M$. This defines an orbifold metric with incomplete edge singularity and
bounded geometry. In a more general setting, any singularity covered by a smooth branched
covering space admits a singular metric of bounded Ricci curvature. \medskip

The previous paragraph provides two explicit examples of spaces, which have bounded Ricci curvature despite having isolated conical or edge singularities. In both cases the singularity arises as an orbifold singularity. Another class of examples for singular spaces with bounded Ricci curvature has been provided by Hein and Sun \cite{HeSu}, who constructed the first examples of compact Ricci flat manifolds with non-orbifold isolated conical singularities. \medskip

Another quite explicit example is the 
case of a knot $\mathbb{S}^1$ embedded into $\mathbb{S}^3$ or any other orientable $3$-manifold.
The normal bundle of $\mathbb{S}^1$ may be equipped with an edge metric of any given angle.
The fibres of the normal bundle are flat two-dimensional cones and the resulting metric, smoothly
extended away from the singularity, is of bounded geometry. \medskip

Let us point out the assumption of a bounded geometry
is obviously satisfied in the geometric setting of $g\!\!\mid_{\U}$ 
being a higher order perturbation of a Ricci-flat incomplete edge metric. 

\subsection{H\"older spaces on singular manifolds} \medskip

\begin{defn}\label{hoelder-A}
The H\"older space $\ho(M\times [0,T]), \A\in (0,1),$ consists of functions 
$u(p,t)$ that are continuous on $\overline{M} \times [0,T]$ with finite $\A$-th H\"older 
norm\footnote{Finiteness of the H\"older norm $\|u\|_{\A}$ in particular implies that $u$ is continuous on the 
closure $\overline{M}$ up to the edge singularity, and the supremum may be taken 
over $(p,p',t) \in \overline{M}^2 \times [0,T]$.}
\begin{align}\label{norm-def}
\|u\|_{\A}:=\|u\|_{\infty} + \sup \left(\frac{|u(p,t)-u(p',t')|}{d_M(p,p')^{\A}+
|t-t'|^{\frac{\A}{2}}}\right) <\infty, 
\end{align}
where the distance function $d_M(p,p')$ between any two points $p,p'\in M$ 
is defined with respect to the incomplete edge metric $g$, and in terms of the local coordinates 
$(x,y,z)$ in the singular neighborhood $\U$ given equivalently by
\begin{align*}
d_M((x,y,z), (x',y',z'))=\left(|x-x'|^2+(x+x')^2|z-z'|^2 + |y-y'|^2\right)^{\frac{1}{2}}.
\end{align*}
The supremum is taken over all $(p,p',t) \in M^2 \times [0,T]$. We also introduce 
the H\"older space of time-independent functions (and suppress $[0,T]$ in the notation)
\begin{align}
\ho(M):= \{u \in \ho(M\times [0,T]) \mid u(\cdot, t) \ \textup{is independent of} \ t\in [0,T]\}.
\end{align}
\end{defn} 

We wish to explain in what way the H\"older space $\ho$ introduced above, may be defined locally.
Consider any finite cover $\{U_i\}_{i\in I}$ of $\overline{M}$ by open coordinate charts
and a partition of unity $\{\phi_j\}_{j\in J}$ subordinate to that cover. We can define 
a H\"older norm by 
\begin{align}\label{partition-hoelder}
\|u\|^{\phi}_{\A}:=\sum_{j\in J} \| \phi_j u \|_{\A}.
\end{align}
Such a norm is equivalent to our original H\"older norm, since for any 
tuple $(p,p') \in M^2$ with distance $d_M(p,p')>\delta$ bounded 
away from zero, the quotient in the second summand of the formula 
\eqref{norm-def} is bounded by $2\delta^{-1}\|u\|_{\infty}$. Consequently,
we may assume without loss of generality that the tuples $(p,p') \in M^2$ are always taken from 
within the same coordinate patch of a given atlas. \medskip

We also need a notion of H\"older spaces with values in the  
vector bundle $S=\textup{Sym}^2({}^{ie}T^*M)$ of symmetric $2$-tensors.
with a fibrewise inner product $g_S$, induced by the Riemannian metric $g$.

\begin{defn}\label{S-0-hoelder}
Denote by $h$ a fibrewise inner product on 
$S=\textup{Sym}^2({}^{ie}T^*M)$ induced by the Riemannian metric $g$.
The H\"older space $\ho (M\times [0,T], S)$ consists of all sections $\w$ of
$S$ which are continuous on $\overline{M} \times [0,T]$, 
such that for any local orthonormal frame 
$\{s_j\}$ of $S$, the scalar functions $g_S(\w,s_j)$ are $\ho (M\times [0,T])$.
\medskip

The $\A$-th H\"older norm of $\w$ is defined using a partition of unity
$\{\phi_j\}_{j\in J}$ subordinate to a cover of local trivializations of $S$, with a 
local orthonormal frame $\{s_{jk}\}$ over $\supp (\phi_j)$ for each $j\in J$. We put
\begin{align}\label{partition-hoelder-2}
\|\w\|^{(\phi, s)}_{\A}:=\sum_{j\in J} \sum_{k} \| g_S(\phi_j \w,s_{jk}) \|_{\A}.
\end{align}
\end{defn}

As before in \eqref{partition-hoelder}, norms corresponding to different choices of $(\{\phi_j\}, \{s_{jk}\})$
are equivalent and we may drop the upper index $(\phi, s)$ from notation.
The supremum norm $\|\w\|_\infty$ is defined similarly. \medskip

We now define the weighted and higher order H\"older spaces.

\begin{defn}\label{funny-spaces}
\begin{enumerate}
\item The weighted H\"older space for $\gamma \in \R$ is
\begin{align*}
&x^\gamma \ho(M\times [0,T], S) := \{ \, x^\gamma \w \mid \w \in \ho(M\times [0,T], S) \, \}, 
\\ &\textup{with H\"older norm} \ \| x^\gamma \w \|_{\A, \gamma} := \|\w\|_{\A}.
\end{align*}
\item The hybrid weighted H\"older space for $\gamma \in \R$ is
\begin{align*}
&\ho_{, \gamma} (M\times [0,T], S) := x^\gamma \ho(M\times [0,T], S)  \, \cap \, 
x^{\gamma + \A} \mathcal{C}^0_{\textup{ie}}(M\times [0,T], S) \\
&\textup{with H\"older norm} \  \| \w \|'_{\A, \gamma} := \|x^{-\gamma} 
\w\|_{\A} + \|x^{-\gamma-\A} \w\|_\infty.
\end{align*}
\item Let $C^k(M\times [0,T],S)$ denote the space of 
$S$-sections that are $k$-times continuously differentiable 
in the open interior $M$ of $\overline{M}$. We identify the local expressions 
$\{x\partial_x, x \partial_y, \partial_z\}$ over $\U$ 
with their smooth extensions to vector fields over $M$.
Then the weighted H\"older spaces of order $k\in \N$ are defined for any 
weight $\gamma \in \R$ as subspaces of 
\begin{equation*}
\begin{split}
&\hok (M\times [0,T], S)_\gamma = \{\w\in \ho_{,\gamma}\cap C^k \mid  \{\V^j \circ \, (x^2 \partial_t)^\ell\} 
\, \w \in \ho_{,\gamma} \ \textup{for any} \ j+2\ell \leq k \}, \\
&\hok (M\times [0,T], S)^b_\gamma = 
\{u\in \ho \cap C^k \mid  \{\V^j \circ \, (x^2 \partial_t)^\ell\} \, u \in 
x^\gamma\ho \ \textup{for any} \ j+2\ell \leq k, \\
&\textup{if there is at least one $(x\partial_x)$ or $\partial_z$ derivative
included in $\V^j \circ \, (x^2 \partial_t)^\ell$}, 
\\ &\textup{otherwise $\V^j \circ \, (x^2 \partial_t)^\ell u = (x\partial_y)^j \circ (x^2 \partial_t)^\ell 
u \in x^{\min\{\gamma, j+2\ell\}} \ho$} \, \}, 
\end{split}
\end{equation*}

\item In case of $\gamma=0$ we just write $\hok = \hok(M\times [0,T], S)^b_0$.

\item The weighted H\"older spaces of time-independent functions
are given by\footnote{Regularity under differentiation by $\partial_t$ becomes irrelevant in this case.}
\begin{align*}
\hok (M, S)_\gamma:= \{u \in \hok (M\times [0,T], S)_\gamma \mid u(\cdot, t) \ \textup{is independent of} \ t\in [0,T]\}, \\
\hok (M, S)^b_\gamma:= \{u \in \hok (M\times [0,T], S)^b_\gamma \mid u(\cdot, t) \ \textup{is independent of} \ t\in [0,T]\}.
\end{align*}

\end{enumerate}
\end{defn} 

In order to define the H\"older norms for $\hok (M\times [0,T], S)_\gamma$
and $\hok (M\times [0,T], S)^b_\gamma$, we consider as before any finite 
cover $\{U_i\}_{i\in I}$ of $\overline{M}$ by open coordinate charts, which we may assume to trivialize $S$
by appropriate refinement, and a partition of unity $\{\phi_j\}_{j\in J}$ subordinate to that cover.
By a small abuse of notation we now identify $\mathcal{V}_e$ with 
a finite set of generating edge vector fields, when applied to sections with 
compact support in $\U$; and write $\mathcal{V}_e$ for any local 
orthonormal frame of vector fields, when applied to sections with compact
support in a coordinate chart with distance bounded from below away from the edge
singularity. We may now introduce $\mathscr{D}:=\{\V^j \circ \, (x^2 \partial_t)^\ell \mid j+2\ell \leq k \}$ and  
can now write the H\"older norms on the higher order H\"older spaces as follows
\begin{equation}\label{norms}
\begin{split}
&\|\w\|_{k+\A, \gamma} = \sum_{j\in J} \sum_{X\in \mathscr{D}} \| X (\phi_j \w) \|'_{\A, \gamma}
+ \|\w\|'_{\A, \gamma}, \quad \textup{on} \ \hok (M\times [0,T], S)_\gamma, \\
&\|\w\|_{k+\A, \gamma} = \sum_{j\in J} \sum_{X\in \mathscr{D}} \| X (\phi_j \w) \|_{\A, \gamma}
+ \|\w\|_{\A}, \quad \textup{on} \ \hok (M\times [0,T], S)^b_\gamma,
\end{split}
\end{equation}
where in the second definition we replace $\| X (\phi_j \w) \|'_{\A, \gamma}$ by 
$\| X (\phi_j \w) \|'_{\A, \min \{\gamma, j+2\ell\}}$ if $X=(x\partial_y)^j \circ (x^2\partial_t)^\ell$.
Any different choice of coordinate charts and the subordinate partition of unity, as well as
different choices of generating vector fields $\mathcal{V}_e$ define equivalent H\"older norms.
\medskip

For sections $\w$ compactly supported away from $\partial \overline{M}$, the 
H\"older norms above are equivalent to the classical parabolic H\"older norms introduced by Ladyzhenskaya, 
Solonnikov and Ural'tseva \cite{LSU}. \medskip

The vector bundle $S$ decomposes into a direct sum of sub-bundles
\begin{align}
S= S_0 \oplus S_1, 
\end{align}
where the sub-bundle $S_0=\textup{Sym}_0^2({}^{ie}T^*M)$
is the space of trace-free (with respect to the fixed metric $g$) symmetric $2$-tensors,
and $S_1$ is the space of pure trace (with respect to the fixed metric $g$) symmetric 
$2$-tensors. The sub bundle $S_1$ is trivial real vector bundle over $M$ of rank 1.
Definition \ref{funny-spaces} extends verbatim to sections of $S_0$ and $S_1$.
Since the sub-bundle $S_1$ is a trivial rank one real vector bundle, its sections
correspond to scalar functions. Hence, we may omit $S_1$ from the notation and
simply write e.g. 
\begin{equation}\begin{split}
&\hok (M\times [0,T])_\gamma := \hok (M\times [0,T], S_1)^b_\gamma, \\
&\hok (M)_\gamma := \hok (M, S_1)^b_\gamma.
\end{split}\end{equation}

The H\"older spaces $\hok (M\times [0,T])^b_\gamma$ and 
$\hok (M\times [0,T], S)_\gamma$ are similar but not the same. They are adapted to the mapping properties
of the heat operators for the Laplace Beltrami operator $\Delta$ and the Lichnerowicz Laplacian $\Delta_L$ 
with the former satisfying stochastic completeness. We will address the analytic reason for using such spaces in Remark \ref{why-such-spaces}. \medskip

Moreover, we refer the reader to the Appendix \ref{spaces-comparison} for a detailed comparison of the various 
H\"older spaces on singular incomplete edge manifolds that appear in the literature, foremost the spaces in \cite{JMR, BV}.
\medskip

We conclude the subsection with a definition of a H\"older regular geometry.

\begin{defn}\label{regular-geometry}
Let $\A\in (0,1), k \in \N_0$ and $ \gamma > 0$. 
An admissible edge space $(M,g)$ is $(\A, \gamma, k)$-H\"older regular if the following two conditions are satisfied
\begin{enumerate}
\item For the curvature $(0,4)$ tensor $R(g)$ acting on any sections
$X_1, X_2, X_3, X_4 \in C^\infty(\overline{M}, {}^{ie}TM)$
$$R(g) [X_1, X_2, X_3, X_4] \in x^{-2} \mathcal{C}^{k,\A}_{\textup{ie}}(M).$$ 
\item $\scal(g) \in x^{-2+\gamma} \mathcal{C}^{k,\A}_{\textup{ie}}(M)$ 
 and the trace-free part of $\textup{Ric}(g)$ is $\mathcal{C}^{k,\A}_{\textup{ie}}(M, S_0)_{-2+\gamma}$.
\end{enumerate}
\end{defn}

\subsection{Existence of the singular Ricci flow} \medskip

Given a compact smooth Riemannian manifold $(M,g_0)$,
the Ricci flow of $g_0$ is by definition a family $g(t), t\in [0,T]$ of Riemannian metrics on $M$, 
satisfying the Ricci flow equation
\begin{equation}\label{RDF-1}
\partial_t g(t) = -2 \, \textup{Ric} (g(t)), \quad g(0) = g_0.
\end{equation}
Ricci flow is not a parabolic equation due to its diffeomorphism invariance. Therefore existence of solutions 
does not follow directly from the classical parabolic theory. This problem is resolved using the so-called 
de Turck trick \cite{Turck}. The de Turck trick leads to an equivalent Ricci de Turck flow $g(t)$, which is given by the following equation.
\begin{equation}\label{RDF-2}
\partial_t g(t) = -2 \, \textup{Ric} (g(t)) + \mathcal{L}_{W(t)} g(t), \quad g(0) = g_0,
\end{equation}
where $W(t)$ is the de Turck vector field defined in terms of the Christoffel symbols for 
the metrics $g(t)$ and a reference metric $\widetilde{g}$\footnote{The reference metric 
$\widetilde{g}$ is often taken as the initial metric $\widetilde{g} = g_0$.} 
\begin{equation}
W(t)^k = g(t)^{ij} \left(\Gamma^k_{ij}(g(t)) - \Gamma^k_{ij}(\widetilde{g})\right). 
\end{equation}
The de Turck vector field $W(t)$ yields a one parameter family of diffeomorphisms $\phi(t)$
and the pullback $\phi(t)^*g(t)$ solves the Ricci flow \eqref{RDF-1}. Ricci de Turck flow is a parabolic equation 
and existence of its solution can be easily obtained by the following argument. The equation \eqref{RDF-2} is linearized by 
writing $g(t) = v(t) + g_0$, which leads to a non-linear parabolic equation for $v(t)$
\begin{equation}\label{RDF-3}
(\partial_t + \Delta_L) v(t) = -2 \, \textup{Ric} (g_0) + O_2(v(t)), \quad v(0) = 0,
\end{equation}
where $\Delta_L$ is the Lichnerowicz Laplacian on symmetric two-tensors and 
$O_2(v(t))$ is a sum of terms which are at least quadratic in $v(t)$ and its first 
and second order derivatives. Clearly, the solution $v(t)$ is a fixed point of the following map 
\begin{equation}
\begin{split}
\Phi: & \ C^{2,\alpha}(M\times [0,T]) \longrightarrow C^{2,\alpha}(M\times [0,T]), \\
&u \mapsto \int_0^t e^{-(t-\wt)\, \Delta_L} \left(
-2 \, \textup{Ric} (g_0) + O_2(u(\wt)) \right) d\wt, 
\end{split}
\end{equation}
where $e^{-t\, \Delta_L}$ is the heat operator of the Lichnerowicz Laplacian $\Delta_L$ and $C^{k,\alpha}$ are the usual 
parabolic H\"older spaces with $\alpha \in (0,1)$ and $k\in \N_0$. Classical Schauder estimates of Ladyzhenskaya, Solonnikov and
Ural'tseva \cite{LSU} essentially prove the mapping property of the heat operator (acting with a convolution in time)
\begin{equation}
\begin{split}
e^{-t\, \Delta_L}: & \ C^{0,\alpha}(M\times [0,T]) \longrightarrow C^{2,\alpha}(M\times [0,T]), \\
&u \mapsto \int_0^t e^{-(t-\wt)\, \Delta_L} u(\wt) d\wt, 
\end{split}
\end{equation}
and imply that $\Phi$ is bounded, since $O_2(u(\wt)) \in C^{0,\alpha}$ for $u \in C^{2,\alpha}$. 
A simple argument shows that for $\mu> 0 $ sufficiently small, $\Phi$ is a contraction on 
\begin{align}
Z_{\mu,T} := \{u \in C^{2,\alpha}(M\times [0,T]) \mid \|u\|_{2,\A} \leq \mu\},
\end{align}
mapping $Z_{\mu,T}$ to itself for $\mu, T > 0$ sufficiently small. Consequently, by the Banach fixed point theorem 
there exists a fixed point $v(t)$ of $\Phi$, which is a solution to the Ricci de Turck flow by construction. \medskip

While on smooth compact manifolds, Ricci flow continues to be a focal point of intensive research, on singular
spaces even existence of Ricci flow is an open problem. If the manifold $(M.g)$ is singular, the 
argument outlined above may break down. The major difficulty hereby is whether
some analogue of parabolic Schauder estimates as derived in \cite{LSU} can be established in the singular setting.
The purpose of the present work is to study the Ricci de Turck flow on a singular edge manifolds.
We derive parabolic Schauder estimates in this setting and prove short time existence of Ricci de Turck flow. 
\medskip

Our first main result establishes short time existence
of Ricci de Turck flow starting at an admissible incomplete edge metric of H\"older regular geometry 
and flowing through the space of singular metrics, which preserves the admissible edge 
structure and H\"older regular geometry. The result holds under an additional assumption of tangential
stability, which is a spectral condition imposed upon the Lichnerowicz Laplace operator 
introduced below in Definition \ref{assump-spectrum} and discussed in detail in Theorem \ref{stability-thm}.
Let us point out that \cite[Theorem 1.4]{KrVe} provides
an extensive list of explicit examples, where tangential stability is satisfied.
 
\begin{thm}\label{main1}
Consider an incomplete edge manifold $(M,g)$ with an admissible edge metric and 
H\"older-regular geometry, satisfying the assumption of tangential stability. Then for short time $g$ may be evolved under 
the Ricci de Turck flow into a family of Riemannian metrics $g(t)$ within the space of admissible edge metrics of 
H\"older regular geometry for some finite time $T>0$. 
\end{thm}

We will also address the relation between the Ricci de Turck and the Ricci flow,
which is intricate in terms of regularity. \medskip

Our second main result concerns Ricci flow starting at metrics that are in a certain sense 
higher order small perturbations of flat incomplete edge metrics. In that case we actually obtain
long time existence. 

\begin{thm}\label{main2}
Consider an incomplete edge manifold $(M,h)$ of H\"older regular geometry 
with an admissible flat edge metric, satisfying the
assumption of tangential stability. If $g_0$ is a higher order sufficiently small perturbation of $h$, then a 
Ricci de Turck flow $g(t)$ of admissible incomplete edge metrics of H\"older regular geometry, 
starting at $g_0$, exists for all time and stays in a small $\varepsilon$-neighborhood of $h$, uniformly in $t\geq 0$.
\end{thm}

In a joint paper with Kr\"oncke \cite{KrVe} we discuss stability
of the Ricci de Turck flow for small perturbations of Ricci flat (not necessarily flat) singular metrics, 
assuming certain integrability conditions outside of the scope of the present paper.\medskip

In fact, Ricci flow through singular metrics has been studied by various authors in dimension two, e.g.
by Mazzeo, Rubinstein and Sesum in \cite{MRS}, our work jointly with Bahuaud in \cite{BV}. 
Somewhat different from the approach taken here, is the work by Giesen and Topping on instantaneously
complete Ricci flow in \cite{Topping} and \cite{Topping2}. Another alternative approach has been 
taken by Miles Simon in \cite{MS}, where Ricci flow smoothens out any Lipschitz singularity instantly.\medskip

The idea of the proof for both Theorem \ref{main1} and Theorem \ref{main2}
is, exactly as in the compact smooth case, to linearize the Ricci de Turck flow 
and apply Banach fixed point theorem in appropriate H\"older spaces. This requires 
mapping properties of the heat operator for the Lichnerowicz Laplacian. The bulk of the
paper is therefore devoted to deriving these mapping properties in the singular setting. 
\medskip

This paper is organized as follows. We begin with the analysis of the Lichnerowicz Laplace 
operator in \S \ref{Lichnerowicz-section} and construct a solution to its heat equation as a polyhomogeneous
conormal distribution on a blown up heat space. In \S \ref{mapping-section} we establish various 
mapping properties of  the heat operator for the Lichnerowicz and the Laplace Beltrami operators. 
We employ these mapping properties to establish existence of a solution to the Ricci de Turck 
flow in \S \ref{short-section} and show in \S \ref{edge-flow} that this flow is indeed a flow of admissible incomplete edge metrics.
Then \S \ref{Ricci-flow} explains how to pass from the Ricci de Turck solution to the corresponding 
solution of the Ricci flow, along with a change in regularity. In \S \ref{Ricci-bounded-section} we 
discuss H\"older regularity of the Ricci de Turck flow for positive times,
an aspect which will be crucial in subsequent maximum principle arguments. We conclude this 
paper with a long time existence result in \S \ref{small} for Ricci flow of metrics that are sufficiently small perturbations
of flat edge metrics. \medskip

\emph{Acknowledgements:} The author thanks Burkhard Wilking, Christoph B\"ohm, Rafe Mazzeo and Eric Bahuaud for 
important discussions about aspects of Ricci flow and encouragement. He thanks Klaus Kr\"oncke for helpful discussions concerning
computations in his paper on Einstein warped products. He is grateful to the anonymous referee for careful reading of the
manuscript, important remarks and suggestions. The author also gratefully acknowledges support
of the Mathematical Institute at M\"unster University.

\section{Lichnerowicz Laplacian on $2$-tensors of an exact cone}\label{Lichnerowicz-section}

In this section we study the rough and the Lichnerowicz Laplace operators 
acting on symmetric $2$-tensors over an exact cone $\mathscr{C}(F):=(0,1) \times F$ 
with an exact conical metric $g=dx^2 \oplus x^2 g_F$. We provide explicit 
formulae and formulate assumptions that are necessary for the subsequent 
analytic arguments.\medskip

Consider for the moment any Riemannian manifold $(M,g)$ of dimension $m$. We will specify
$M$ to be an exact cone $\mathscr{C}(F)$ right after the general definition.
Let $L$ denote any vector bundle associated to $T^*M$, for instance the bundle 
$S$ of symmetric trace-free $2$-tensors $\textup{Sym}^2_0(T^*M)$. Let $\nabla$ denote the induced 
Levi-Civita connection acting on smooth compactly supported sections as 
\begin{align*}
&\nabla: C^\infty_0(M, L) \to C^\infty_0(M,L\otimes T^*M), \\
&\nabla^2: C^\infty_0(M, L) \to C^\infty_0(M,L\otimes T^*M\otimes T^*M).
\end{align*}
The rough Laplacian $\Delta$, acting on smooth compactly supported sections of $L$, is then 
defined as follows. Consider the pointwise inner product on fibres of $L$, induced by the Riemannian 
metric $g$ on $M$. Let $\{e_i\}_{i = 1}^m$ denote a local orthonormal frame of $TM$, where 
$m$ is the dimension of $M$. The rough Laplacian is given by
\begin{align}
\Delta := - \sum_{i = 1}^m \nabla^2_{e_i e_i} \equiv - \textup{tr}(\nabla^2).
\end{align}
The Lichnerowicz Laplacian on symmetric covariant $2$-tensors is defined in terms of the 
rough Laplacian $\Delta$ and additional curvature terms by
\begin{align}
\Delta_L:= \Delta + 2(\textup{Ric} - \textup{Riem}),
\end{align}
where for any symmetric covariant $2$-tensor $\w$ on any Riemannian manifold $(M,g)$, 
with the corresponding curvature tensor $\textup{Riem}(g)$ and the Ricci curvature tensor $\textup{Ric}(g)$, we have
\begin{equation}
\begin{split}
&(\textup{Ric} \, \w)_{ij} := \left(\textup{Ric}(g)_{ik} \, g^{k\ell} \w_{\ell j} + \textup{Ric}(g)_{jk} \, g^{k\ell} \w_{\ell i}\right)/2, \\
&(\textup{Riem} \, \w)_{ij} := \textup{Riem}(g)_{ikj\ell} \, g^{ks} g^{\ell t} \w_{st}.
\end{split}
\end{equation}
Let us now specify the action of $\Delta_L$ in case of an exact cone.
Let $M=\mathscr{C}(F):=(0,1) \times F$ be an exact cone with an exact conical metric 
$g=dx^2 \oplus x^2 g_F$. Let $L=\textup{Sym}^2({}^{ie}T^*\mathscr{C}(F))$ 
be the bundle of symmetric covariant $2$-tensors on the cone. We may decompose 
$$
C^\infty_0(\mathscr{C}(F), L)  \equiv 
C^\infty_0(\mathscr{C}(F), \textup{Sym}^2({}^{ie}T^*\mathscr{C}(F))) 
= C^\infty_0(\mathscr{C}(F)) g \oplus C^\infty_0(\mathscr{C}(F),S_0)
$$ 
into the pure trace and the trace-free parts with respect to the 
Riemannian edge metric $g$. This decomposition is preserved under the Lichnerowicz Laplacian. By a minor abuse of
notation, $\Delta_L$ shall refer to the Lichnerowicz Laplacian on the trace-free part, while $\Delta'_L$ denotes its action on 
the pure trace component. $\Delta'_L$ is given by the action of the Laplace Beltrami operator with 
($f=\dim F$)
\begin{align} 
\Delta'_L ( u\cdot g) = \left(- \partial_x^2 u - \frac{f}{x} \partial_x u  + \frac{1}{x^2} \, \square'_L u \right) \cdot g,
\end{align}
where $\square'_L$ is the Laplace Beltrami operator of $(F,g_F)$. The action of $\Delta_L$ on the trace-free
component has been computed by Delay \cite[Lemma 4.2]{Delay} and Guillarmou, Moroianu and Schlenker  \cite[7.4]{Gui}.
Let $\{z_\alpha\}$ denote local coordinates on $F$. We write for any symmetric trace-free $2$-tensor $\w \in C^\infty_0(\mathscr{C}(F),S_0)$
\begin{align*}
\w_{xx}:= \w(\partial_x, \partial_x), \quad &\w_{x\A}:= \w(\partial_x, x^{-1}\partial_{z_\A}), 
\qquad \w_{\A \beta}:= \w(x^{-1}\partial_{z_\A}, x^{-1}\partial_{z_\beta}); \\
&\w'_{x\A}:= \w(\partial_x, \partial_{z_\A}) = x \, \w_{x\A}, 
\quad  \w'_{\A \beta}:= \w(\partial_{z_\A}, \partial_{z_\beta})= x^2 \w_{\A \beta}.
\end{align*}
For simplicity of notation we denote the scalar function $\w_{xx}$
by $\eta$, the $(1,0)$-tensor $(\w'_{x\A})_\A$ by $\xi'$, and the symmetric $2$-tensor $(\w'_{\A \beta})_{\A\beta}$ 
on the cross section of the cone by $\kappa'$. Clearly, our convention is to use 
greek letters for components corresponding to the cross section $F$. Operators and quantities referring
to the cross section $(F,g_F)$ of the cone $\mathscr{C}(F)$ are denoted with an 
additional index $F$. Then we obtain as in \cite[Lemma 4.2]{Delay} and \cite[7.4]{Gui}
\begin{align*}
&(\Delta_L \w)_{xx} = \left( -\partial_x^2 - \frac{f}{x} \partial_x + \frac{1}{x^2} \left(\Delta^F_L + 2f+2\right)\right) \eta
- \frac{4}{x^3} \delta_F \xi', \\
&(\Delta_L \w)'_{x\A} = \left( -\partial_x^2 - \frac{f-2}{x} \partial_x + \frac{1}{x^2} \left(\Delta^F_L + f+2\right)\right) \xi'_\A
- \frac{2}{x} \partial_{z_\A}\eta + \frac{2}{x^3}\delta_F \kappa'_\A, \\
&(\Delta_L \w)'_{\A\beta} = \left( -\partial_x^2 - \frac{f-4}{x} \partial_x + \frac{1}{x^2} \left(\Delta^F_L - 4 \right)\right) \kappa'_{\A\beta}
- 2 \, \eta \, g^F_{\A\beta} \\ & \qquad \qquad \qquad \qquad \qquad \qquad \qquad \qquad 
+ \frac{2}{x^2} (\textup{tr}_{g_F} \kappa') g^F_{\A\beta} - \frac{2}{x} \delta^*_F \xi'_{\A\beta},
\end{align*}
where $\Delta^F_L$ acting on $h$ refers to the Lichnerowicz Laplacian acting on symmetric $2$-forms
over $F$ and moreover, we have introduced the following notation
\begin{align*}
&\Delta^F_L \eta := \Delta_F \eta \equiv \square'_L \eta, \quad \delta_F \xi' := g_F^{\A\beta} (\nabla^F_\beta \xi')_\A, \\
&\Delta^F_L \xi'_\A := (\nabla_F^*\nabla_F \xi')_\A + g_F^{\gamma\beta} \textup{Ric}^F_{\beta\A} \xi'_\gamma, \quad
\delta_F \kappa'_\A := g_F^{\gamma\beta} (\nabla^F_\beta \kappa')_{\A\gamma}, \\
&(\delta^*_F \xi')_{\A\beta}:= (\nabla^F_\A \xi')_\beta + (\nabla^F_\beta \xi')_\A.
\end{align*}
The formulae become more transparent if we switch to the action on $(\w_{xx}, \w_{x\A}, \w_{\A \beta})$.
We denote the $(1,0)$-tensor $(\w_{x\A})_\A$ by $\xi$, and the symmetric $2$-tensor $(\w_{\A \beta})_{\A\beta}$ 
on the cross section of the cone by $\kappa$. Clearly, $\xi' = x \xi$ and $\kappa'=x^2 \kappa$. 
The action of the Lichnerowicz Laplacian with respect to that rescaling is now given by
\begin{align*}
&(\Delta_L \w)_{xx} = \left( -\partial_x^2 - \frac{f}{x} \partial_x + \frac{1}{x^2} \left(\Delta^F_L + 2f + 2 \right)\right) \eta
- \frac{4}{x^2} \delta_N \xi, \\
&(\Delta_L \w)_{x\A} = \left( -\partial_x^2 - \frac{f}{x} \partial_x + \frac{1}{x^2} \left(\Delta^F_L + 4 \right)\right) \xi_\A
- \frac{2}{x^2} \partial_{z_\A} \eta + \frac{2}{x^2}\delta_F \kappa_\A, \\
&(\Delta_L \w)_{\A\beta} = \left( -\partial_x^2 - \frac{f}{x} \partial_x + \frac{1}{x^2} \left(\Delta^F_L + 2 - 2f \right)\right) \kappa_{\A\beta}
- \frac{2}{x^2} \, \eta \, g^F_{\A\beta} \\ & \qquad \qquad \qquad \qquad \qquad \qquad \qquad \qquad 
+ \frac{2}{x^2} (\textup{tr}_{g_F} \kappa) g^F_{\A\beta} - \frac{2}{x^2} \delta^*_F \xi_{\A\beta},
\end{align*}
Identifying any trace-free $\w$ with the vector of its 
$(\w_{xx}, \w_{x\A}, \w_{\A\beta})$ components,  
\begin{equation}
\begin{split}
C^\infty_0(\mathscr{C}(F), S_0) &\to C^\infty_0((0,1), C^\infty(F) 
\times \Omega^1(F) \times \textup{Sym}^2(T^*F)),\\
\w &\mapsto (\w_{xx}, \w_{x\A}, \w_{\A\beta}),
\end{split}
\end{equation}
where $\Omega^1(F)$ denotes differential $1$-forms on $F$,
we arrive at the following expression for the action of the Lichnerowicz Laplacian on $\w \equiv (\w_{xx}, \w_{x\A}, \w_{\A\beta})$
\begin{align}\label{reg-sing}
\Delta_L \w =  \left(- \partial_x^2 - \frac{f}{x} \partial_x  + \frac{1}{x^2} \, \square_L \right)  \w,
\end{align}
where the action of $\square_L$ is given by
\begin{align}\label{tangential-matrix}
\square_L  = \left( \begin{array}{clc} \Delta^F_L +2f+2 & -4 \delta_F & 0 \\ -2 \partial_z & \Delta^F_L+4 & 2 \delta_F \\ 
-2g_F & -2 \delta_F^* & \Delta^F_L+2-2f+ 2g_F \textup{tr}_{g_F} \end{array}\right).
\end{align}
We may now introduce the assumption of tangential stability.

\begin{defn}\label{assump-spectrum}
We call an admissible edge manifold $(M,g)$ tangentially stable with lower bounds
$u_0, u_1>0$, if $\min (\textup{Spec} \, 
\square_L) = u_0$
and $\min (\textup{Spec} \, \square'_L \backslash \{0\}) = u_1$.
\end{defn}

In the follow-up joint work with Kr\"oncke \cite[Theorem 1.3]{KrVe} we characterize
tangential stability explicitly in terms of the spectrum of the Einstein and the Hodge 
Laplace operators on the cross section $(F,g_F)$. We state the result here for
completeness. 

\begin{thm}\label{stability-thm}
	Let $(F,g_F)$ be a compact Einstein manifold of dimension $f\geq 3$ 
	with the Einstein constant $(f-1)$. We write $\Delta_E$ for its Einstein operator, and denote the Laplace Beltrami 
	operator by $\Delta$. Then tangential stability holds if and only if $\mathrm{Spec}(\Delta_E|_{TT})>0$ and $\mathrm{Spec}(\Delta)\setminus \left\{0\right\}\cap (f, 2(f+1)]=\varnothing$. 
\end{thm}

We also identify in \cite[Theorem 1.4]{KrVe}
an extensive list of explicit examples, where tangential stability is satisfied.
This includes e.g. certain simple Lie groups and rank-1 symmetric spaces of compact type. 
The actual statement in \cite{KrVe} also identifies the cases where tangential
stability fails. Moreover \cite{KrVe} shows that the only example where $(F,g_F)$
is weakly tangentially stable in the sense of Definition \ref{weak-stability-definition},
but not tangentially stable is the case of a sphere.

\section{Heat operator of the Lichnerowicz Laplacian} 
We consider the Lichnerowicz Laplacian $\Delta_L$, acting on trace-free 
symmetric two-tensors on an admissible 
incomplete edge manifold $(M,g)$. In this subsection we
consider the homogeneous or the inhomogeneous heat equations 
\begin{equation}\label{heat-hom-inhom}
\begin{split}
&(\partial_t + \Delta_L) \, \w_{\textup{hom}}(t,p)  = 0, \ \w_{\textup{hom}}(0,p)= \w_0(p), \\
&(\partial_t + \Delta_L) \, \w_{\textup{inhom}}(t,p)  = v(t,p), \ \w_{\textup{inhom}}(0,p)= 0, \\
\end{split}
\end{equation}
and obtain their fundamental solutions, following the heat kernel construction of \cite{MazVer}. 
The solutions are given in terms of an integral convolution operator acting on compactly supported sections such that
\begin{equation} \label{eqn:hk-on-functions}
\begin{split}
&\w_{\textup{hom}}(t,p) = \left(e^{-t\Delta_L} \w_0\right) (t,p) := \int_M \left(e^{-t\, \Delta_L} \left(p,\widetilde{p} \right),  
\w_0(\widetilde{p})\right)_g \dv (\widetilde{p}), \\
&\w_{\textup{inhom}}(t,p) = \left(e^{-t\Delta_L} v\right) (t,p) := \int_0^t \int_M \left(e^{-(t-\wt)\, \Delta_L} \left( p,\widetilde{p} \right),  
v(\wt, \widetilde{p})\right)_g \dv (\widetilde{p}) d\wt.
\end{split}
\end{equation}
In both cases we denote the fundamental solution by $e^{-t\Delta_L}$, which acts by time convolution on 
time-dependent sections. \medskip

Under the additional assumption $\Delta_L\geq 0$ on smooth compactly supported 
sections of $S_0$, the fundamental solution $e^{-t\Delta_L}$ can be identified
with the heat operator of the Friedrichs self-adjoint extension of the Lichnerowicz Laplacian.
This will be explained below in Theorem \ref{essential} and is crucial 
later on for the argument on the long time existence of the Ricci flow 
starting at small perturbations of flat metrics. 

\subsection{Heat kernel of a model operator} 

Before we proceed with an asymptotic analysis of the fundamental solution for the 
Lichnerowicz Laplacian $\Delta_L$ on an admissible edge manifold $(M,g)$, we consider 
a model operator which already comprises all the central properties of $\Delta_L$. 
Let us write $\R_+:= (0,\infty)$. Consider for any $\mu\geq 0$ the model operator 
\begin{align}\label{model-l}
\ell_\mu := - \frac{d^2}{ds^2} - \frac{f}{s} \frac{d}{ds} + \frac{1}{s^2}
\left(\mu^2 - \left(\frac{f-1}{2}\right)^2\right): C^\infty_0(\R_+) \to C^\infty_0(\R_+),
\end{align}
acting on compactly supported smooth test functions $C^\infty_0(\R_+)$.
This operator is symmetric with respect to the inner product of $L^2(\R_+, s^f ds)$.
It can be conveniently studied under the unitary rescaling transformation 
$\Phi: L^2(\R_+, s^f ds) \to L^2(\R_+, ds)$ with $\Phi (u) = s^{f/2} u$. 
Then
\begin{align}
\Phi \circ \ell_\mu \circ \Phi^{-1} = - \frac{d^2}{ds^2} + \frac{1}{s^2}
\left(\mu^2 - \frac{1}{4}\right) =: L_\mu,
\end{align}
is a symmetric operator in $L^2(\R_+,ds)$.
Consider the maximal and minimal domains for $L_\mu$ acting on $C^\infty_0(\R_+)$
\begin{equation}
\begin{split}
\dom_{\max}(L_\mu) := \{ &\w \in L^2(\R_+) \mid L_\mu \w \in L^2(\R_+)\}, \\
\dom_{\min}(L_\mu) := \{ &\w \in \dom_{\max}(L_\mu) \mid \exists (\w_n) \subset C^\infty_0(\R_+):
\\ &\w_n \xrightarrow{L^2} \w, \quad L_\mu \w_n \xrightarrow{L^2} L_\mu \w\},
\end{split}
\end{equation}
where $L_\mu \w \in L^2(\R_+)$ on $\w \in L^2(\R_+)$ is understood in the distributional sense. 
The maximal (minimal) domain of $\ell_\mu$ is defined similarly and the domains are related by the unitary transformation $\Phi$
$$
\dom_{\max}(\ell_\mu) = \Phi^{-1} \dom_{\max}(L_\mu), \quad \dom_{\min}(\ell_\mu) = \Phi^{-1} \dom_{\min}(L_\mu).
$$
By explicit computations, see for example \cite[Proposition 2.10]{Ver}, any $\w \in \dom_{\max}(L_\mu)$ admits an asymptotic expansion 
\begin{equation}
\begin{split}
&\w = c^+_\mu (\w) \, x^{\frac{1}{2}} + c^-_\mu (\w) \, x^{\frac{1}{2}} \log (x) + \widetilde{\w},
\quad \textup{if} \ \mu = 0, \\ 
&\w = c^+_\mu (\w) \, x^{\mu + \frac{1}{2}} + c^-_\mu (\w) \, x^{-\mu + \frac{1}{2}} + \widetilde{\w},
\quad \textup{if} \ \mu \in (0,1), \\ 
&\w =  \widetilde{\w},
\quad \textup{if} \ \mu \geq 1, \\ 
\end{split}
\end{equation}
with coefficients $c^\pm_\mu (\w) \in \C$ and $\widetilde{\w} \in \dom_{\min}(L_\mu)$. 
Hence any $\w \in \dom_{\max}(\ell_\mu) = \Phi^{-1} \dom_{\max}(L_\mu) $ admits an asymptotic expansion
\begin{equation}\label{max-expansion-model}
\begin{split}
&\w = c^+_\mu (\w) \, x^{ - \frac{(f-1)}{2}} + c^-_\mu (\w) \, x^{ - \frac{(f-1)}{2}} \log(x) + \widetilde{\w},
\quad \textup{if} \ \mu = 0, \\ 
&\w = c^+_\mu (\w) \, x^{\mu - \frac{(f-1)}{2}} + c^-_\mu (\w) \, x^{-\mu - \frac{(f-1)}{2}} + \widetilde{\w},
\quad \textup{if} \ \mu \in (0,1), \\ 
&\w =  \widetilde{\w},
\quad \textup{if} \ \mu \geq 1, \\ 
\end{split}
\end{equation}
with coefficients $c^\pm_\mu (\w) \in \C$ and $\widetilde{\w} \in \dom_{\min}(\ell_\mu)$. We set $c^\pm _\mu(\w) = 0$ for $\mu \geq 1$.
For any $\w, v \in  \dom_{\max}(\ell_\mu)$ we compute using integration by parts
\begin{equation}\label{int-by-parts}
\begin{split}
&\langle L_\mu \Phi \w, \Phi v \rangle_{L^2(\R_+, ds)} - \langle \Phi  \w, L_\mu \Phi v \rangle_{L^2(\R_+, ds)} =  \\
& \langle \ell_\mu  \w,   v \rangle_{L^2(\R_+, s^f ds)} - \langle   \w, \ell_\mu   v \rangle_{L^2(\R_+, s^f ds)} 
= c_\mu \left( c^+_\mu(\w) \overline{c^-_\mu(v)} - c^-_\mu(\w) \overline{c^+_\mu(v)} \right).
\end{split}
\end{equation}
where $c_\mu = -1$ for $\mu = 0$ and $c_\mu = 2\mu$ otherwise. 
From this formula it becomes clear that boundary conditions need to be imposed 
on the coefficients $c^\pm_\mu(\w)$ in order to obtain a self-adjoint 
extension of $\ell_\mu$ and $L_\mu$ In case $\mu \geq 1$, $\ell_\mu$ and $L_\mu$ are essentially self-adjoint, since no 
boundary terms appear after integration by parts in \eqref{int-by-parts}.
\medskip

Existence of a self-adjoint extension of $L_\mu$ and $\ell_\mu$ with the same lower
bound as $L_\mu$ and $\ell_\mu$, respectively, both acting on $C^\infty_0(\R_+)$, is due to Friedrichs and 
Stone, see Riesz and Nagy \cite[Theorem on p. 330]{RN}, who introduced the so-called 
Friedrichs self-adjoint extension. Providing the functional analytic construction of the Friedrichs 
extension is out of scope of the present discussion. However the Friedrichs extension 
$L^{\mathscr{F}}_\mu$ of $L_\mu$, as well as the Friedrichs extension 
$\ell^{\mathscr{F}}_\mu$ of $\ell_\mu$ can be explicitly characterized as follows
\begin{equation}\label{model-cone-domain}
\begin{split}
&\dom (L^{\mathscr{F}}_\mu) = \{\w \in \dom_{\max}(L_\mu) \mid c^-_\mu(\w) = 0\}, \\
&\dom (\ell^{\mathscr{F}}_\mu) = \{\w \in \dom_{\max}(\ell_\mu) \mid c^-_\mu(\w) = 0\}.
\end{split}
\end{equation}

Both extensions are related by the unitary transformation 
$$
L^{\mathscr{F}}_\mu = \Phi \circ \ell^{\mathscr{F}}_\mu \circ \Phi^{-1}, \quad
\dom (\ell^{\mathscr{F}}_\mu) = \Phi^{-1} \dom (L^{\mathscr{F}}_\mu).
$$
The heat kernel $H_\mu$ of the Friedrichs extension $L^{\mathscr{F}}_\mu$ is well-known and 
 \cite[Proposition 2.3.9]{Les:OOF} provided its explicit expression in 
terms of the modified Bessel function $I_\mu$ of first kind
\begin{align}
H_\mu(t, s, \widetilde{s}) = \frac{1}{2t}(s\widetilde{s})^{1/2}I_{\mu}
\left(\frac{s\widetilde{s}}{2t}\right)e^{-\frac{s^2+\widetilde{s}^2}{4t}}. 
\end{align}
Hence, the heat kernel $e^{-t\ell_\mu}$ of $\ell^{\mathscr{F}}_\mu$ is given by 
$(s\widetilde{s})^{-f/2}H_\mu(t, s,\widetilde{s})$, so that we obtain
\begin{align}\label{model-heat-kernel}
e^{-t\ell_\mu} (s, \widetilde{s}) = \frac{1}{2t}(s\widetilde{s})^{(1-f)/2}I_{\mu}
\left(\frac{s\widetilde{s}}{2t}\right)e^{-\frac{s^2+\widetilde{s}^2}{4t}}. 
\end{align}

\subsection{Microlocal construction of a fundamental solution} 

The Lichnerowicz Laplacian writes in local coordinates $(x,y,z)$ in the 
singular neighborhood $\U$, which is locally a fibration of cones 
$\mathscr{C}(F)$ over $B$, as a sum of the Lichnerowicz Laplacian 
$\Delta^{\mathscr{C}}_L$ on the cone $\mathscr{C}(F)$
and the Lichnerowicz Laplacian in $y \in \R^b$, plus higher order terms.  \medskip

The fundamental solution $e^{-t\Delta_L}$ will be a distribution on $M^2_h=\R^+\times \overline{M}^2$, 
taking values in $S_0\boxtimes S_0$, which 
is a vector bundle over $\overline{M}^2$ with the fibre $S_{0,p}\times S_{o,q}$ for any $(p,q) \in \overline{M}^2$.
Consider the local coordinates near the corner in $M^2_h$ given by $(t, (x,y,z), (\widetilde{x}, \wy, \widetilde{z}))$, 
where $(x,y,z)$ and $(\widetilde{x}, \wy, \widetilde{z})$ are two copies of coordinates on $\overline{M}$ near the boundary. 
The kernel $e^{-t\Delta_L}(t, (x,y,z), (\wx,\wy,\wz))$ has non-uniform behaviour at the submanifolds
\begin{align*}
&A =\{ (t, (x,y,z), (\wx,\wy,\wz))\in M^2_h \mid t=0, \, x=\wx=0, \, y= \wy\}, \\
&D =\{ (t, p, \widetilde{p})\in M^2_h \mid t=0, \, p=\widetilde{p}\},
\end{align*}
which requires an appropriate blowup of the heat space $M^2_h$, 
such that the corresponding heat kernel lifts to a polyhomogeneous distribution 
in the sense of the following definition, which we cite from \cite{Mel:TAP} and \cite{MazVer}.

\begin{defn}\label{phg}
Let $\mathfrak{W}$ be a manifold with corners and $\{(H_i,\rho_i)\}_{i=1}^N$ an enumeration 
of its (embedded) boundaries with the corresponding defining functions. For any multi-index $J= (b_1,
\ldots, b_N)\in \C^N$ we write $\rho^J := \rho_1^{b_1} \ldots \rho_N^{b_N}$.  Denote by 
$\mathcal{V}_b(\mathfrak{W})$ the smooth vector fields on $\mathfrak{W}$ lying
tangent to all boundary faces. \medskip

\begin{enumerate}
\item A distribution $\w$ on $\mathfrak{W}$ is said to be conormal,
if $\w$ is a restriction of a distribution across the boundary faces of $\mathfrak{W}$, 
$\w\in \rho^J L^\infty(\mathfrak{W})$ for some $J\in \C^N$ and $V_1 \ldots V_\ell \w \in \rho^J L^\infty(\mathfrak{W})$
for all $V_j \in \mathcal{V}_b(\mathfrak{W})$ and for every $\ell \geq 0$.  \medskip

\item An index set $E_i = \{(\gamma,p)\} \subset {\mathbb C} \times {\mathbb N_0}$ 
satisfies the following hypotheses:

\begin{enumerate}
\item $\textup{Re}(\gamma)$ accumulates only at $+\infty$,
\item for each $\gamma$ there exists $P_{\gamma}\in \N_0$, such 
that $(\gamma,p)\in E_i$ for all $p \leq P_\gamma$,
\item if $(\gamma,p) \in E_i$, then $(\gamma+j,p') \in E_i$ for all $j \in {\mathbb N_0}$ and $0 \leq p' \leq p$. 
\end{enumerate} \.  \medskip

\item An index family $E = (E_1, \ldots, E_N)$ is an $N$-tuple of index sets.  \medskip

\item Finally, we define the notion of polyhomogeneous conormal distributions iteratively in the dimension of $\mathfrak{W}$.
We say that a conormal distribution $\w$ is polyhomogeneous on $\mathfrak{W}$ 
with index family $E$, we write $\w\in \mathscr{A}_{\textup{phg}}^E(\mathfrak{W})$, 
if $\w$ is conormal and if in addition, near each $H_i$, 
\[
\w \sim \sum_{(\gamma,p) \in E_i} a_{\gamma,p} \rho_i^{\gamma} (\log \rho_i)^p, \ 
\textup{as} \ \rho_i\to 0,
\]
with coefficients $a_{\gamma,p}$ conormal on $H_i$, polyhomogeneous with index $E_j$
at any intersection $H_i\cap H_j$ of hypersurfaces. In the first iteration step, where 
$\dim \mathfrak{W} = 1$ and each boundary hypersurface $H_i$ is given by a point,
the coefficients $a_{\gamma,p}$ are complex numbers.
\end{enumerate}
\end{defn}

Blowing up submanifolds $A$ and $D$ is a geometric procedure of introducing polar coordinates on $M^2_h$, 
around the submanifolds together with the minimal differential structure which turns polar coordinates into smooth
functions on the blowup. A detailed account on the blowup procedure is given e.g. in \cite{Mel:TAP} and \cite{Gr}. 
Here we only give a basic idea and refer the reader to these references for an explicit account. \medskip

First we blow up parabolically 
(i.e. we treat $\sqrt{t}$ as a smooth variable) the submanifold $A$.
This defines $[M^2_h, A]$ as the disjoint union of
$M^2_h\backslash A$ with the interior spherical normal bundle of $A$ in $M^2_h$,
equipped with the minimal differential structure 
such that smooth functions in the interior of $M^2_h$ and polar coordinates 
on $M^2_h$ around $A$ are smooth. The interior spherical normal bundle of $A$ defines a new boundary 
hypersurface $-$ the front face ff in addition to the previous boundary faces 
$\{x=0\}, \{\wx=0\}$ and $\{t=0\}$, which lift to rf (the right face), lf (the left face) and 
tf (the temporal face), respectively.  \medskip

The actual heat-space $\mathscr{M}^2_h$ is obtained by a second parabolic blowup of  
$[M^2_h, A]$ along the diagonal $D$, lifted to a submanifold of $[M^2_h, A]$. 
We proceed as before by cutting out the lift of $D$ and replacing it with its spherical 
normal bundle, which introduces a new boundary face $-$ the temporal diagonal td. 
The heat-space $\mathscr{M}^2_h$ comes with the blowdown map $\beta: \mathscr{M}^2_h \to \R^+ \times \overline{M}^2$,
which is a diffeomorphism from the interior of $\mathscr{M}^2_h$ onto $(0,\infty) \times M^2$.
The heat space $\mathscr{M}^2_h$ and the lift of a curve starting at the corner of $M^2_h$ is illustrated in Figure \ref{heat-incomplete}. 
The base point of the lifted curve at the front face indicates the angle under which the curve approaches the corner 
of $(0,\infty) \times M^2$ before the lift. \medskip

\begin{figure}[h]
\begin{center}
\begin{tikzpicture}[scale=1.2]

\draw[->] (0,-0.2) -- (0,2);
\draw[->] (0,-0.2) -- (-2,-1);
\draw[->] (0,-0.2) -- (2,-1);

\node at (-0.4,1.8) {$t$};
\node at (-2.05, -0.6) {$x$};
\node at (2.05, -0.6) {$\wx$};

\draw (5,0.7) -- (5,2);
\draw(4.3,-0.5) -- (3,-1);
\draw (5.7,-0.5) -- (7,-1);
\draw (5,0.7) .. controls (4.5,0.6) and (4.3,0) .. (4.3,-0.5);
\draw (5,0.7) .. controls (5.5,0.6) and (5.7,0) .. (5.7,-0.5);
\draw (4.3,-0.5) .. controls (4.5,-0.6) and (4.6,-0.7) .. (4.7,-0.7);
\draw (5.7,-0.5) .. controls (5.5,-0.6) and (5.4,-0.7) .. (5.3,-0.7);
\draw (4.7,-0.7) .. controls (4.7,-0.3) and (5.3,-0.3) .. (5.3,-0.7);
\draw (4.7,-1.4) .. controls (4.7,-1) and (5.3,-1) .. (5.3,-1.4);
\draw (5.3,-0.7) -- (5.3,-1.4);
\draw (4.7,-0.7) -- (4.7,-1.4);

\draw[thick, dotted] (0,-0.2) .. controls (0.5,1) and (2,1) .. (2,2);
\draw[thick, dotted] (5.3,0.2) .. controls (5.5,1) and (6.8,1) .. (7.3,2);

\node at (5, -1.7) {$\td$};
\node at (5,0.1) {$\ff$};
\node at (6,0.3) {$\rf$};
\node at (4,0.3) {$\lf$};
\node at (5.8, -1) {$\tf$};
\node at (4.2, -1) {$\tf$};
\node at (2.5, 1) {$\beta$};

\draw[<-] (2, 0.5) -- (3,0.5);

\end{tikzpicture}
\end{center}
\caption{The heat-space $\mathscr{M}^2_h$.}
\label{heat-incomplete}
\end{figure}

We now describe projective coordinates in a neighborhood of the front face ff in $\mathscr{M}^2_h$, which are used often 
as a convenient replacement for the polar coordinates. The drawback it that projective coordinates are 
not globally defined over the entire front face. Near the top corner of the front face ff, projective coordinates are given by
\begin{align}\label{top-coord}
\rho=\sqrt{t}, \  \xi=\frac{x}{\rho}, \ \widetilde{\xi}=\frac{\wx}{\rho}, \ u=\frac{y-\wy}{\rho}, \ z, \ \wy, \ \wz.
\end{align}
With respect to these coordinates, $\rho, \xi, \widetilde{\xi}$ are in fact the defining 
functions of the boundary faces ff, rf and lf respectively. For the bottom right corner of the front face, projective coordinates are given by
\begin{align}\label{right-coord}
\tau=\frac{t}{\wx^2}, \ s=\frac{x}{\wx}, \ u=\frac{y-\wy}{\wx}, \ z, \ \wx, \ \wy, \ \widetilde{z},
\end{align}
where in these coordinates $\tau, s, \widetilde{x}$ are
the defining functions of tf, rf and ff respectively. 
For the bottom left corner of the front face,
projective coordinates are obtained by interchanging 
the roles of $x$ and $\widetilde{x}$. We illustrate some of these projective 
coordinates in the Figure \ref{heat-incomplete2}. \medskip

\begin{figure}[h]
\begin{center}
\begin{tikzpicture}[scale=1.2]

\draw (5,0.7) -- (5,2);
\draw(4.3,-0.5) -- (3,-1);
\draw (5.7,-0.5) -- (7,-1);
\draw (5,0.7) .. controls (4.5,0.6) and (4.3,0) .. (4.3,-0.5);
\draw (5,0.7) .. controls (5.5,0.6) and (5.7,0) .. (5.7,-0.5);
\draw (4.3,-0.5) .. controls (4.5,-0.6) and (4.6,-0.7) .. (4.7,-0.7);
\draw (5.7,-0.5) .. controls (5.5,-0.6) and (5.4,-0.7) .. (5.3,-0.7);
\draw (4.7,-0.7) .. controls (4.7,-0.3) and (5.3,-0.3) .. (5.3,-0.7);
\draw (4.7,-1.4) .. controls (4.7,-1) and (5.3,-1) .. (5.3,-1.4);
\draw (5.3,-0.7) -- (5.3,-1.4);
\draw (4.7,-0.7) -- (4.7,-1.4);

\node at (5, -1.7) {$\td$};
\node at (5,0.1) {$\ff$};

\draw[thick, dotted, <-] (4.95,0.85) .. controls (4.5,0.8) and (4.2,0.3) .. (4.1,-0.5);
\draw[thick, dotted, ->] (5.05,0.85) .. controls (5.5,0.8) and (5.8,0.3) .. (5.9,-0.5);
\draw[thick, dotted, ->] (4.2,-0.6) .. controls (4.3,-1) and (5.7,-1) .. (5.8,-0.6);

\node at (5.4,1.8) {$\rho$};
\node at (4.5,0.95) {$\xi$};
\node at (6.1,-0.2) {$\tau$};
\node at (5.8, -1) {$s$};

\node at (7,0.95) {$\rf = \{\xi = 0\}$};
\node at (7.16,0.55) {$= \{s= 0\}$};
\node at (2.9,0.2) {$\ff = \{\rho = 0\}$};

\end{tikzpicture}
\end{center}
\caption{The heat-space $\mathscr{M}^2_h$.}
\label{heat-incomplete2}
\end{figure}

Projective coordinates 
on $\mathscr{M}^2_h$ near temporal diagonal are given by 
\begin{align}\label{d-coord}
\eta=\frac{\sqrt{t}}{\wx}, \ S =\frac{(x-\wx)}{\sqrt{t}}, \ 
U= \frac{y-\wy}{\sqrt{t}}, \ Z =\frac{\wx (z-\wz)}{\sqrt{t}}, \  \wx, \ 
\wy, \ \widetilde{z}.
\end{align}
In these coordinates, tf is defined as the limit $|(S, U, Z)|\to \infty$, 
ff and td are defined by $\widetilde{x}, \eta$, respectively. 
The blow-down map $\beta: \mathscr{M}^2_h\to M^2_h$ is in 
local coordinates simply the coordinate change back to 
$(t, (x,y, z), (\widetilde{x},\wy, \widetilde{z}))$. \medskip

The fundamental solution $e^{-t\Delta_L}$ of the Lichnerowicz Laplacian $\Delta_L$ is constructed exactly as in \cite{MazVer}. 
In order to indicate the basic idea, consider the lift $\beta^* (x^2 \Delta_L)$ of $x^2 \Delta_L$
to the heat-space $\mathscr{M}^2_h$. This amounts to writing the differential operator e.g. in 
projective coordinates \eqref{right-coord}. The restriction of $\beta^* (x^2 \Delta_L)$ to the front face 
does not differentiate in $\wy \in B$, which can be viewed as a parameter\footnote{In fact ff is a fibration over $B$ and
we consider the restriction of $\beta^* (x^2 \Delta_L)$ to the fibres of ff. } and is given 
for each fixed $\wy$ by the so-called normal operator
\begin{align}\label{product-heat}
N(x^2\Delta_L)_{\wy} := \beta^* (x^2 \Delta_L) \restriction \ff 
= s^2 \left( \Delta^{\mathscr{C}}_{L,\wy} + \Delta^{\R^b}_{L,\wy}\right),
\end{align}
where $\Delta^{\mathscr{C}}_{L}$ is the Lichnerowicz Laplacian on the model cone 
$\mathscr{C}(F)$ acting in the variables $(s,z)$ and defined with respect to the metric
$ds^2 + s^2 g_F$. The other summand $\Delta^{\R^b}_{L,\wy}$ is the Lichnerowicz 
Laplacian acting in the variable $u \in \R^b \cong T_{\wy} B$, 
defined with respect to the metric $g^B(\wy)$ on $T_{\wy} B$. \medskip

The heat kernel of $\Delta^{\R^b}_{L,\wy}$ acting on $u\in \R^b$ is constructed in the classical way. 
The heat kernel of $\Delta^{\mathscr{C}}_{L}$ on the model cone is
obtained as follows. $\Delta^{\mathscr{C}}_L$ is computed in \eqref{reg-sing} and
reduces over $\lambda$-eigenspaces $E_\lambda= \langle \phi_\lambda \rangle$ of the tangential operator 
$\square_L$ to a scalar multiplication operator, acting on $u \phi_\lambda$ with $u\in C^\infty_0(0,1)$ by 
\begin{align}
\Delta^{\mathscr{C}}_L \left( u \phi_\lambda \right)
= \left( -\partial^2_s - \frac{f}{s}\partial_s + \frac{\lambda}{s^2}
\right) u \cdot \phi_\lambda =: (\ell_\mu u) \phi_\lambda,
\end{align}
with $\mu^2 = \lambda + (f-1)^2/4 \geq 0$. 
Consequently, the heat kernel of $\Delta^{\mathscr{C}}_L$ is given by the following sum
\begin{align}\label{Bessel-heat-spin}
e^{-t\Delta^{\mathscr{C}}_L} (s,z, \widetilde{s}, \wz) = \sum_{\mu(\lambda)} e^{-t\ell_{\mu(\lambda)}} (s,\widetilde{s})
\,  \phi_\lambda(z) \otimes \phi_\lambda(\wz). 
\end{align}
From the explicit expression \eqref{model-heat-kernel}  we obtain 
its asymptotics as $s\to 0$ \footnote{By symmetry the same asymptotics holds as $\widetilde{s} \to 0$.}
\begin{equation}\label{model-expansion}
\begin{split}
e^{-t\Delta^{\mathscr{C}}_L} (s,z, \widetilde{s}, \wz) &\sim
\sum_{\lambda \, \in \, \textup{Spec} \, \square_L} a_{\lambda,k}(t, \widetilde{s}, z, \wz) \, s^{ \, \mu(\lambda) - \frac{(f-1)}{2} + 2k} \\
&\sim \sum_{\lambda  \, \in \, \textup{Spec} \, \square_L} \sum_{k=0}^\infty
a_{\lambda,k}(t, \widetilde{s}, z, \wz) \, s^{\sqrt{\lambda +  \left(\frac{f-1}{2}\right)^2} - \frac{(f-1)}{2} + 2k}
 \end{split}
\end{equation}
Due to the direct sum decomposition in \eqref{product-heat}, the heat equation at the 
front face admits a fundamental solution $N(e^{-t\Delta_L})_{\wy}$ obtained exactly 
as in \cite[(3.10)]{MazVer} as a direct sum of the heat kernel for $\Delta^{\mathscr{C}}_{L}$ and the heat kernel
of $\Delta^{\R^b}_{L,\wy}$ 
\begin{align}\label{Bessel-heat-spin}
N(e^{-t\Delta_L}) (\tau, s,z,\wz,u)_{\wy}  := e^{-\tau \, \Delta^{\mathscr{C}}_{L}}  (s, z, \widetilde{s}=1, \wz) 
\, e^{-\tau \, \Delta^{R^b}_{L,\wy}} (u,\widetilde{u}=0).
\end{align}
In order to construct the fundamental solution, the normal operator $N(e^{-t\Delta_L})_{\wy}$ 
is extended off the front face and corrected iteratively, which involves composition 
of Schwartz kernels on $\mathscr{M}^2_h$. Following the heat kernel construction 
in \cite{MazVer} verbatim, we arrive at the following result. 

\begin{thm}\label{heat-asymptotics}
Let $(M,g)$ be an incomplete edge manifold of dimension $m$ with an 
admissible edge metric $g$. Then the Lichnerowicz Laplacian $\Delta_L$ acting on symmetric trace-free 2-tensors 
admits a fundamental solution $e^{-t\Delta_L}$ to its heat equation, such that
the lift $\beta^*e^{-t\Delta_L}$ is a polyhomogeneous function on $\mathscr{M}^2_h$ 
taking values in $S_0 \boxtimes S_0$ with $S_0 = \textup{Sym}^2_0({}^{ie}T^*M)$ and the index sets 
$(-m+\N_0, 0)$ at ff, $(-m+\N_0, 0)$ at td, vanishing to infinite order at tf. The index set at rf and lf is given 
explicitly by $E+\N_0$ where
\begin{align}\label{E-index-definition}
E=\left\{ \mu = \sqrt{\lambda + \left(\frac{f-1}{2}\right)^2} - \left(\frac{f-1}{2}\right) 
\mid \lambda \in \textup{Spec}(\square_L)\right\}.
\end{align}
\end{thm}

For convenience of the reader, let us note that e.g. in projective coordinates \eqref{right-coord} 
and \eqref{d-coord} the index sets at the boundary faces ff, rf as well as td, indicate the following asymptotic expansions

\begin{equation}
\begin{split}
& \beta^* e^{-t\Delta_L} 
\sim \sum_{j=0}^\infty a_j(\tau, s, z, \wz, u, \wy) \,  \wx^{-m+j}, \ \textup{as} \ \wx \to 0, \\
& \beta^* e^{-t\Delta_L} 
\sim \sum_{j=0}^\infty b_j(S, U, Z, \wx, \wy, \wz) \, \eta^{-m+j}, \ \textup{as} \ \eta \to 0, \\
& \beta^* e^{-t\Delta_L} 
\sim \sum_{\lambda  \, \in \, \textup{Spec} \, \square_L} \sum_{j=0}^\infty
c_{\lambda,j}(\tau, \widetilde{s}, z, \wz) \, s^{\sqrt{\lambda +  \left(\frac{f-1}{2}\right)^2} - \frac{(f-1)}{2} + j},
\ \textup{as} \ s \to 0.
 \end{split}
\end{equation}
While we do not repeat the heat kernel construction of \cite{MazVer} here, let us indicate some fundamental 
reasons for the asymptotics at the various boundary faces. The negative leading order $(-m)$ of asymptotics 
of the fundamental solution at the front face (as $\wx \to 0$) is a consequence of the fact that the heat kernel 
$H$ on an the model edge $\mathscr{C}(F) \times \R^b$ with metric $dx^2+ x^2 g_F + dy^2$ is 
homogeneous of order $(-m)$ with ($r>0$)
$$
H(r^2 t, rx, z, r\wx, \wz, r(y-\wy)) = r^{-m} H(t,x, z, \wx, z, (y-\wy)).
$$ 
The negative leading order $(-m)$ of asymptotics
of the fundamental solution at the temporal diagonal (as $\eta \to 0)$ is due to the initial conditions at $t=0$ of the 
homogeneous heat equation in \eqref{heat-hom-inhom}. Finally, the expansion at rf comes from the asymptotics \eqref{model-expansion}.

\begin{remark}\label{isospectrality}
In case of the $2$-tensor $g_F$ on $\partial M$ restricting to a smooth variable family of Riemannian metrics on fibres $F$,
that is not necessarily isospectral with respect to the tangential operator $\square_L$,
the arguments and the main statements of this work continue to hold. However, the statement of Theorem \ref{heat-asymptotics}
has to be adapted in that case: the fundamental solution $e^{-t\Delta_L}$ still admits a polyhomogeneous 
expansion at the front face ff and the temporal diagonal td of same order as before, however has a rather complicated behaviour at lf and rf.
\end{remark}

Similar result in \cite{MazVer} constructs the heat kernel of the Laplace Beltrami
operator on $(M,g)$ as a polyhomogeneous function on $\mathscr{M}^2_h$ with an index set
$E'+ \N_0$ at rf and lf, defined similarly in terms of the spectrum of $\square'_L$.

\begin{remark}
Tangential stability introduced in Definition \ref{assump-spectrum} is in fact 
equivalent to asking for a lower bound of $E$ and $E'$. More precisely, 
the minimal elements $\mu_0 \in E$ and $\mu_1 \in E'\backslash \{0\}$ are given by 
\begin{equation}\label{minimal-elements-definition}
\mu_0 := \sqrt{u_0 + \left(\frac{f-1}{2}\right)^2} - \left(\frac{f-1}{2}\right), \quad
\mu_1 := \sqrt{u_1 + \left(\frac{f-1}{2}\right)^2} - \left(\frac{f-1}{2}\right).
\end{equation}
Clearly, $\mu_0 = \min (E+\N_0) = \min E_\rf$, however in general $\mu_1$ need not be equal 
to the minimum of the index set $E'_\rf\backslash \{0\} = (E'+\N_0)\backslash \{0\}$. In fact, $\mu_1= \min 
E'_\rf\backslash \{0\}$ without any restrictions only if $\mu_1 \in (0,1]$. Otherwise, $\mu_1$
is the minimal element of the index set $E'_\rf$ of the 
heat kernel for the Laplace Beltrami operator at rf and lf, 
if the edge is a sufficiently higher order perturbation of
a trivial fibration of exact cones. To make this precise, recall the notation of 
Definition \ref{admissible}. Assume that $\partial M \cong F \times B$ and 
the fibration $\phi: (\partial M, g_F \oplus g_B) \to (B,g_B)$ is the obvious projection 
onto the second factor. Assume that in the tubular neighborhood 
$\U \cong \mathscr{C}(F) \times B$ of the edge singularity, the edge metric is given by 
$g=\overline{g} + h$ with 
\begin{equation}
\overline{g} \mid_{\U} = dx^2 \oplus x^2 g_F \oplus g_B,
\end{equation}
and $h$ a symmetric $2$-tensor such that $|h|_{\overline{g}} = O(x^N)$ as $x\to 0$
for some $N \in \N$. Let us write $\N_N:= \{0\} \cup \{q \in \N \mid q \geq N\}$. Then 
$E'_\rf = E'+ \N_N$, since $|h|_{\overline{g}} = O(x^N)$ as $x\to 0$. Consequently,
$\mu_1 = \min E'_\rf \backslash \{0\}$ if $|h|_{\overline{g}} = O(x^N)$ as 
$x\to 0$ with $N \geq \mu_1$.
\end{remark}

We conclude the section with an observation that assuming
non-negativity of the Lichnerowicz Laplacian $\Delta_L$
acting on symmetric trace-free 2-tensors, the 
fundamental solution in Theorem \ref{heat-asymptotics} is the
heat operator corresponding to the Friedrichs self-adjoint extension 
of $\Delta_L$. 

\begin{thm}\label{essential}
Assume that $(M,g)$ is tangentially stable and  
$\Delta_L$ acting on $C^\infty_0(M,S_0)$ is non-negative. 
Then the Friedrichs self-adjoint extension of $\Delta_L$ is non-negative
as well and the fundamental solution $e^{-t\Delta_L}$ is the corresponding 
heat operator. 
\end{thm}

\begin{proof}
Consider the maximal and minimal domains for $\Delta_L$ acting on $C^\infty_0(M,S_0)$
\begin{equation}
\begin{split}
\dom_{\max}(\Delta_L) := \{ &\w \in L^2(M,S_0) \mid \Delta_L \w \in L^2(M,S_0)\}, \\
\dom_{\min}(\Delta_L) := \{ &\w \in \dom_{\max}(\Delta_L) \mid \exists (\w_n) \subset C^\infty_0(M,S_0):
\\ &\w_n \xrightarrow{L^2} \w, \quad \Delta_L \w_n \xrightarrow{L^2} \Delta_L \w\}.
\end{split}
\end{equation}
Exactly as worked out in the joint work of the author with Mazzeo \cite[Lemma 2.2]{MazVer}, 
see also \eqref{max-expansion-model} for the explicit model cone case, 
any $\w \in \dom_{\max}(\Delta_L)$ admits a weak asymptotic expansion 
\begin{equation}
\w = \sum_{\mu(\lambda) \in E \cap [0,1)} c^+_\lambda (\w) x^{\mu(\lambda) - \frac{(f-1)}{2}} + 
\sum_{\mu(\lambda) \in E \cap [0,1)} c^-_\lambda (\w) x^{-\mu(\lambda) - \frac{(f-1)}{2}} + \widetilde{\w},
\end{equation}
where $x\in (0,1)$, $E$ is the index set defined in \eqref{E-index-definition}, 
$\widetilde{\w} \in \dom_{\min}(\Delta_L)$ and the coefficients $c^\pm_\lambda$ are 
of negative regularity, i.e. there is an expansion of the pairing $\int_B \w(x,y,z) \phi(y) dy$ for any 
test function $\phi \in C^\infty(B)$. The expansion above is simpler than the one in 
\cite[Lemma 2.2]{MazVer} due tangential stability, so that each 
$\mu(\lambda) \in E \cap [0,1) \equiv E \cap [\mu_0, 1)$ is positive. \medskip

Assuming that $\Delta_L$ acting on $C^\infty_0(M,S_0)$ is non-negative, symmetric and
densely defined, there exists its Friedrichs self-adjoint extension $\Delta_L^{\mathscr{F}}$ with the same lower
bound, which is due to Friedrichs and Stone, see Riesz and Nagy \cite[Theorem on p. 330]{RN}. 
The domain of $\Delta_L^{\mathscr{F}}$ has been identified in \cite[Proposition 2.5]{MazVer} 
by specifying coefficients $c^\pm_\lambda (\w)$ as follows (see \eqref{model-cone-domain} for the 
explicit model cone case)
\begin{equation}
\dom(\Delta_L^{\mathscr{F}}) = \{ \w \in \dom_{\max}(\Delta_L) \mid c^-_\lambda (\w) = 0, 
\ \textup{for} \ \mu(\lambda) \in E \cap [0,1)\}.
\end{equation}
Using the asymptotic description of the Schwartz kernel for $e^{-t\Delta_L}$, one finds
that the fundamental solution maps into $\dom(\Delta_L^{\mathscr{F}})$ for any fixed $t>0$.
Now a verbatim repetition of the arguments for \cite[Proposition 3.4]{MazVer}
proves the statement.
 \end{proof}

We remark that by following the argument of Gell-Redmann and Swoboda \cite[Proposition 13]{Swoboda}
one may deduce that $\Delta_L$ is essentially self-adjoint if $u_0 >\dim F$ from the mapping properties 
of the fundamental solution. This can be intuitively expected, since the condition  $u_0 >\dim F$ translates to 
$\mu_0 > 1$ in view of \eqref{minimal-elements-definition}, and the operators $\ell_\mu$ are in the
limit point case at $x=0$ for $\mu \geq \mu_0 > 1$.

\section{Mapping properties of the Lichnerowicz heat operator}\label{mapping-section}

We continue under the assumption of tangential stability introduced in Definition \ref{assump-spectrum}
and study $\Delta_L$ acting on 
trace-free symmetric $2$-tensors $\w \in C^\infty_0(M\times [0,T], S_0)=\textup{Sym}_0^2({}^{ie}T^*M)$.
We denote its fundamental solution (also referred to as the heat operator) by $H$.
We also fix any $\delta>0$. Our main result in this section is the following theorem.

\begin{thm}\label{main-est1} 
Consider an edge manifold $(M,g)$ with an admissible 
edge metric $g$ satisfying tangential stability as in Definition \ref{assump-spectrum}. 
Consider the index set at the right and left face as in Theorem \ref{heat-asymptotics}, 
with the minimal element $\mu_0 > 0$. Fix any $\gamma \in (1-\dim F, \mu_0)$ and 
any H\"older exponent $\A=1/N \in (0, \mu_0-\gamma)\cap (0,1)$,
for $N\in \N$ sufficiently large. Recall the notation $S_0=\textup{Sym}_0^2({}^{ie}T^*M)$. 
Then the Lichnerowicz heat operator defines a bounded mapping between weighted H\"older spaces
$($for any $\varepsilon \in (0,1])$
\begin{equation}\label{two-mappings}
\begin{split}
&H: \hok (M\times [0,T], S_0)_{-2+\gamma} \to \mathcal{C}^{k+2,\A}_{\textup{ie}}(M\times [0,T], S_0)_\gamma, \\
&H: \mathcal{C}^{k+1,\A}_{\textup{ie}}(M\times [0,T], S_0)_{-2+\gamma+\varepsilon} \to 
t^{\frac{\varepsilon}{2}} \mathcal{C}^{k+2,\A}_{\textup{ie}}(M\times [0,T], S_0)_{\gamma},
\end{split}
\end{equation}
where $t^{\frac{\varepsilon}{2}} \mathcal{C}^{k+2,\A}_{\textup{ie}, \gamma}
= \{ \, \w \mid \exists \, u \in  \mathcal{C}^{k+2,\A}_{\textup{ie}, \gamma} :  \w(t,p) = t^{\frac{\varepsilon}{2}} u(t,p)\}$.
\end{thm}

\begin{proof} 
The statement is proved using the microlocal properties of the heat kernel lifted to
the blowup space $\mathscr{M}^2_h$, where $\rho_*$ shall denote a defining function of the boundary face 
$*$ in $\mathscr{M}^2_h$. \medskip

We mimic a similar statement in \cite{BV} which is proved using stochastic completeness for the heat kernel of the 
Laplace Beltrami operator. We consider here the Lichnerowicz Laplacian and 
are not aware of any equivalent of stochastic completeness 
on tensors. This requires to some extent different analytic arguments. 
We do not write out the argument for the second statement, 
which follows by similar estimates, since better $x$-weight and higher regularity of the starting
space yield additional $(\rho_\td \rho_\ff)^\varepsilon \leq C \sqrt{t}^\varepsilon$. \medskip

When performing the estimates we will use Corollary \ref{mean-cor} and pretend notationally
that $F$ and $B$ are one-dimensional. The estimates in the general case 
are performed verbatim. Moreover, we will always denote uniform positive constants appearing 
in our estimates by $C>0$, even though they might differ from estimate to estimate. We only use 
derivatives in space, since due to the heat equation, 
differentiation of the heat kernel in time can be replaced by the Lichnerowicz Laplacian. Finally,
we will assume without loss of generality that $k=0$. General $k\in \N$ affects estimates near td,
where we may pass derivatives of the heat kernel to derivatives of the section using integration by parts.
\medskip

Let $X$ denote the operator acting on symmetric $2$-tensors by multiplication with the radial function $x:\U \to [0,1)$, extended 
smoothly to $M$ such that $x \restriction M\backslash \U \geq 1$. Then the first mapping in \eqref{two-mappings}
with $k=0$, is bounded if and only if 
\begin{align}\label{without-weights}
X^{-\gamma} \circ e^{-t\Delta_L} \circ X^{-2+\gamma}: 
\ho (M\times [0,T], S_0)_{0} \to \mathcal{C}^{2,\A}_{\textup{ie}}(M\times [0,T], S_0)_0.
\end{align}
Consider any $v \in \ho(M\times [0,T], S_0)_{0}$. 
Then in view of the Definition \ref{funny-spaces}, the map in \eqref{without-weights} is bounded if and only if
for $\w := X^{-\gamma} \circ e^{-t\Delta_L} \circ X^{-2+\gamma} v$ the
supremum and the H\"older norms
$$
\| X^{-\A}\w \|_\infty, \| X^{-\A} \V \w \|_\infty, \| X^{-\A} \V^2 \w \|_\infty, \| \w \|_\A, \| \V \w \|_\A, \| \V^2 \w \|_\A, 
$$
are bounded by the norm of $v$, up to a constant that depends only on $\Delta_L$. 
Here $\V$ refers to differentiation given by edge vector fields. It suffices to bound 
$\| X^{-\A} \V^2 \w \|_\infty$ and $\| \V^2 \w \|_\A$, since the other norms with only first order
or no differentiation at all, are estimated along the same lines. \medskip

Set $G:= X^{-\gamma} \circ \V^2 e^{-t\Delta_L} \circ X^{-2+\gamma}$. Then by definition
$\| \V^2 \w \|_\A$ is bounded if and only if in addition to the supremum norms, 
for any local coordinate patch $U$, which is also a trivializing neighborhood of $S_0$,
we have an estimate of the form
\begin{equation}\label{Hoelder-estimate}
\sup\limits_{U\times [0,T]} \frac{\left\| G v (p,t) - G v (p',t')\right\|}{d_M(p,p')^\A + |t-t'|^{\frac{\A}{2}}}  
 \leq C \| v \|'_{\A,0},
\end{equation}
for some uniform constant $C>0$, where we have refined the atlas of $M$, such that any 
coordinate patch $U$ is a trivializing neighborhood of $S_0$ and the tuples $(p,p')\in M^2$ 
lie inside the same coordinate patch $U$. Note that \eqref{Hoelder-estimate} holds if the following
two estimates hold
\begin{equation}\label{Hoelder-estimate2}
\begin{split}
&\left\| G v (p,t) - G v (p',t)\right\| \leq C \| v \|'_{\A,0} \, d_M(p,p')^\A, \\
&\left\| G v (p,t) - G v (p,t')\right\| \leq C\| v \|'_{\A,0} \, |t-t'|^{\frac{\A}{2}}.
\end{split}
\end{equation}
The proof is now structured as follows. In \S \ref{proof1} we prove the first estimate 
of \eqref{Hoelder-estimate2}, the so-called H\"older estimate in space. In \S \ref{proof2} we prove the second estimate 
of \eqref{Hoelder-estimate2}, the so-called H\"older estimate in time. In \S \ref{proof3} we 
estimate the supremum norm $\| X^{-\A} \V^2 \w \|_\infty$. At various steps in the estimates we are motivated by 
the corresponding estimates in \cite{LSU}.

\subsection{H\"older differences in space}\label{proof1}

We can always arrange for either $x \restriction U\geq 1/2$ or $U \subset \U$ being of the form $U \cong (0,1) \times Y$
with $Y\subset \partial M$ and $S_0\restriction U \cong (0,1) \times S_Y$. 
We model our estimates after a similar analysis in \cite{BV} and begin by introducing a notation 
\begin{align*}
&U_+:= \{\widetilde{p} \in U \mid d_M(p, \widetilde{p}) \leq d_M(p,p')\}, \\
&U_-:= \{\widetilde{p} \in U \mid d_M(p, \widetilde{p}) \geq d_M(p,p')\}. 
\end{align*}
Given a coordinate patch $U$ with $x \restriction U \geq 1/2$, which trivializes $S_0$ by assumption,
we extend the restriction $v(p)$ of the section $v$ to the fibre over $p\in U$ 
to a constant function over $U$. Otherwise $U \subset \U$ is the form $U \cong (0,1) \times Y$
and $S_0\restriction U \cong (0,1) \times S_Y$ and for $p=(x,y,z) \in U$ we extend the restriction 
$v(x,y,z) \in S_p$ to all of $U$ constantly only in the $(y,z) \in Y$ direction. We may now write
\begin{align*}
&G v (p,t) - G v (p',t) \\
&= \int_0^t \int_{U_+} (G(t-\wt, p,\widetilde{p}) - 
G(t-\wt, p',\widetilde{p})) (v(\wt, \widetilde{p})-v(\wt, p)) \, d\wt \, \textup{dvol}_g(\widetilde{p})\\
&+ \int_0^t \int_{U_-} (G(t-\wt, p,\widetilde{p}) - 
G(t-\wt, p',\widetilde{p})) (v(\wt, \widetilde{p})-v(\wt, p)) \, d\wt \, \textup{dvol}_g(\widetilde{p})
\end{align*}
\begin{align*}
& + \int_0^t \int_{U} (G(t-\wt, p,\widetilde{p}) - 
G(t-\wt, p',\widetilde{p})) \, v(\wt, p) \, d\wt \, \textup{dvol}_g(\widetilde{p})\\
& + \int_0^t \int_{M\backslash U} (G(t-\wt, p,\widetilde{p}) - 
G(t-\wt, p',\widetilde{p})) \, v(\wt, \widetilde{p}) \, d\wt \, \textup{dvol}_g(\widetilde{p})\\
& \qquad \qquad \qquad =: L_1 + L_2 + L_3 + L_4.
\end{align*}
Note that the endomorphism $G(\cdot, \widetilde{p})$ can be applied to the vector $v(\wt, p)$
only for $\widetilde{p}$ and $p$ lying in the same coordinate patch $U$ with the corresponding
local trivialization of the vector bundle $S_0$. This explains why we have separated
out the integral $L_4$. 

\begin{remark}\label{why-such-spaces}
At this point we would like to explain the reason for the definition of spaces
\eqref{funny-spaces} with different weights assigned to the H\"older and the supremum norms.
Recall that $v \in \ho_{,0} = \ho \, \cap \, x^{\A} \mathcal{C}^0_{\textup{ie}}$. The
terms $L_1$ and $L_2$ contain differences of $v$, and hence using H\"older regularity 
$v \in \ho$ one obtains an improvement by $\rho_\ff^\A$ in the estimates of the integrands at the 
front face. However, in the terms $L_3$ we do not have differences of $v$ and hence H\"older regularity
of $v$ does not play a role in the estimates. Rather, we use $v \in x^{\A} \mathcal{C}^0_{\textup{ie}}$
and the $x^\A$-weight still provides an improvement by $\rho_\ff^\A$ in the estimates of the integrand at the 
front face. 
\end{remark}

The second term $L_2$ is now estimated exactly as the term $I_3$ in \cite[\S 3.1]{BV}.
In fact the estimates here are even easier using Corollary \ref{mean-cor} and better front face behaviour. \medskip

We rewrite the first term $L_1$ as follows
\begin{align*}
L_1  &= \int_0^t \int_{U_+}  G(t-\wt, p,\widetilde{p}) (v(\wt, \widetilde{p}) - v(\wt, p))  \, d\wt \, \textup{dvol}_g(\widetilde{p})\\
& - \int_0^t \int_{U_+}  G(t-\wt, p',\widetilde{p}) (v(\wt, \widetilde{p}) - v(\wt, p'))  \, d\wt \, \textup{dvol}_g(\widetilde{p})\\
& - \int_0^t \int_{U_+}  G(t-\wt, p',\widetilde{p}) (v(\wt, p') - v(\wt, p))  \, d\wt \, \textup{dvol}_g(\widetilde{p}) \\
& \qquad \qquad \qquad =: L_{11} + L_{12} + L_{13}.
\end{align*}
Both, $L_{11}$ and $L_{12}$ are estimated precisely as the term $I_1$ in \cite[\S 3.1]{BV}.
The third summand $L_{13}$ is estimated as the term $I_4$ in \cite[\S 3.1]{BV}. \medskip

It remains to estimate the terms $L_3$ and $L_4$ from above. We begin with the easier term $L_4$. Recall that 
for estimating the H\"older norm, we may assume without loss of generality that the two fixed points 
$p=(x,y,z)$ and $p'=(x',y',z')$ lie in the coordinate neighborhood $U$. Since the estimates away from 
the singular neighborhood $\U$ are classical, we may also assume that $U \subset \U$. Then we may write 
\begin{align*}
L_4 &= |x-x'| \int_0^t\int_{M\backslash U} \partial_\xi G (t-\wt, \xi, y, z, \widetilde{p}) \, 
v(\wt, \widetilde{p}) \, d\wt \, \textup{dvol}_g(\widetilde{p}) \\ 
&+ |y-y'| \int_0^t\int_{M\backslash U} \partial_\gamma G (t-\wt, x', \gamma, z, \widetilde{p}) \, 
v(\wt, \widetilde{p}) \, d\wt \, \textup{dvol}_g(\widetilde{p}) \\ 
&+ |z-z'| \int_0^t\int_{M\backslash U} \partial_\zeta G (t-\wt, x', y', \zeta, \widetilde{p}) \, 
v(\wt, \widetilde{p}) \, d\wt \, \textup{dvol}_g(\widetilde{p}).
\end{align*}
where $(\xi, \gamma, \zeta)$ is a point on the straight connecting line between $(x,y,z)$ and $(x',y',z')$.
Assume that $\widetilde{p}=(\wx, \wy, \wz) \in \U \backslash U$. Then by construction, the distance $d_B(Y,\wy)$
between $Y$ and $\wy$ is uniformly bounded from below for any $Y\in (y, y', \gamma)$. Consequently, we find  
for the integrands in the various coordinate systems \eqref{top-coord}, \eqref{right-coord} and \eqref{d-coord} that 
for some uniform positive constant $C>0$ we have
\begin{align*}
\rho^{-1}_\ff = |u| d_B(Y,\wy)^{-1} \leq C |u|, (\rho_\ff \rho_\td)^{-1} = |U| d_B(Y,\wy)^{-1} \leq C |U|. 
\end{align*}
Since the heat kernel is bounded as $|u|$ and $|U|$ tend to infinity, we conclude that each integrand 
above vanishes to infinite order at the front and temporal diagonal faces. Consequently $L_4$ may be 
bounded in terms of the supremum norm of $v$ and $d_M(p,p')$ up to some uniform constant. 
If $\widetilde{p}\notin \U$, then the heat kernels in the integrals above are supported away from the front 
and temporal diagonal faces in $\mathscr{M}_h^2$, so that the estimates are classical in the same spirit as before.
\medskip

It remains to estimate $L_3$ which occupies the remainder of the subsection. It is here that we need to use
Corollary \ref{mean-cor}. Note that while the previous estimates employed H\"older regularity of $v$, estimation 
of $L_3$ uses only the $x^\A \mathcal{C}^0_{\textup{ie}}$ bound of 
$v \equiv x^\A \w \in \ho(M\times [0,T], S_0)_{0} \subset x^\A \mathcal{C}^0_{\textup{ie}}$.
Hence we consider an operator $G':= G \circ X^\A$ of the following asymptotics 
\begin{align}\label{G-primed}
\beta^*G' = \rho_\ff^{-m-2+\A} \rho_\td^{-m-2} \rho_\rf^{\mu_0-\gamma} 
\rho_\lf^{-2+\mu_0+\gamma+\A} \rho_\tf^\infty \mathcal{G},
\end{align}
where $\mathcal{G}$ is a bounded polyhomogeneous distribution on $\mathscr{M}^2_h$.
The notation $\rho_\tf^\infty$ indicates that the kernel $\beta^*G'$ is vanishing to infinite order at 
the temporal face tf and hence the equation \eqref{G-primed} holds with $\rho_\tf^\infty$ 
replaced by $\rho_\tf^K$ for any $K \in \N$. 
We obtain from Corollary \ref{mean-cor} for $\A = 1/N$ with $N$ sufficiently large
\begin{align*}
&|L_3| \, d_M((x,y,z), (x',y',z'))^{-\A}  
\\ & \leq \| \int_0^t\int_U \xi^{\frac{N-1}{N}} \partial_\xi G' (t-\wt, \xi, y, z, \wx, \wy, \wz) 
\, \w(\wt, x,y,z) \, d\wt \, \textup{dvol}_g(\wx, \wy, \wz) \, \| 
\\ & + \| \int_0^t\int_U \|\gamma -\wy \|^{\frac{N-1}{N}} 
\partial_\gamma G' (t-\wt, x', \gamma, z, \wx, \wy, \wz) \, \w(\wt, x,y,z) \, 
d\wt \, \textup{dvol}_g(\wx, \wy, \wz) \, \|  
\\ & + \| \int_0^t\int_U x'^{-\frac{1}{N}} \partial_\zeta G' (t-\wt, x', y', \zeta, \wx, \wy, \wz) 
\, \w(\wt, x,y,z) \, d\wt \, \textup{dvol}_g(\wx, \wy, \wz) \, \| \\
&\qquad \qquad \qquad \qquad \qquad \qquad \qquad \qquad \qquad \qquad \qquad \qquad \qquad =: I_1 + I_2 + I_3,
\end{align*}
We assume $x<x'$ without loss of generality, otherwise just rename the variables. 
In view of the heat kernel asymptotics established in Theorem \ref{heat-asymptotics} and in view of the particular 
fact that the heat kernel is exponentially vanishing for $\frac{\|\gamma -\wy \|}{\rho_\ff}$ going to infinity, we find
\begin{equation}\label{XH}
\begin{split}
\beta^*(\xi^{\frac{N-1}{N}} \partial_\xi G') &= \rho_\ff^{-m-2} \rho_\td^{-m-3} \rho_\rf^{\mu_0-\gamma-\A} 
\rho_\lf^{-2+\mu_0+\gamma+\A} \rho_\tf^\infty  G_1, \\
\beta^*(\|\gamma -\wy \|^{\frac{N-1}{N}} \partial_\gamma G') &= 
\rho_\ff^{-m-2} \rho_\td^{-m-3} \rho_\rf^{\mu_0-\gamma} 
\rho_\lf^{-2+\mu_0+\gamma+\A} \rho_\tf^\infty  G_2, \\
\beta^*(x'^{-\frac{1}{N}} \partial_\zeta G') &= 
\rho_\ff^{-m-2} \rho_\td^{-m-3} \rho_\rf^{\mu_0-\gamma-\A} 
\rho_\lf^{-2+\mu_0+\gamma+\A} \rho_\tf^\infty  G_3, \\
\end{split}
\end{equation}
where the kernels $G_1,G_2$ and $G_3$ are uniformly bounded at all boundary faces of the heat space blowup $\mathscr{M}^2_h$.
We proceed with estimates of $I_1,I_2$ and $I_3$ by assuming that the heat kernel is compactly supported near the 
corresponding corners of the front face in $\mathscr{M}^2_h$. In order to deal with each integral in a uniform notation,
we write $X:=\xi$, when dealing with $I_1$ and $X:=x'$ otherwise. We write $Y:=y$ when dealing with $I_1$,
$Y:=\gamma$ when dealing with $I_2$ and $Y:= y'$ when dealing with $I_3$. Similarly, we write $Z:= \zeta$
when dealing with $I_3$ and $Z:=z$ otherwise. \medskip

For the purpose of brevity, we omit the estimates at the top corner of ff 
and just point out that the estimates are parallel to those near the lower right corner
with same front face behaviour. We write out the optimal estimates which yield additional weights.

\subsubsection{Estimates near the lower left corner of the front face:}
Let us assume that the heat kernel $H$ is compactly supported near the lower left corner of the front face. 
Its asymptotic behaviour is appropriately described in the following projective coordinates
\begin{align}\label{left-coord}
\tau=\frac{t-\wt}{X^2}, \ s=\frac{\wx}{X}, \ u=\frac{Y-\widetilde{y}}{X}, \ X, \ Y, \ Z, \ \widetilde{z},
\end{align}
where in these coordinates $\tau, s, x$ are the defining functions of tf, lf and ff respectively. 
The coordinates are valid whenever $(\tau, s)$ are bounded as $(t-\wt,x,\wx)$ approach zero. 
For the transformation rule of the volume form we compute
\begin{align*}
\beta^*(d\wt \dv(\wx, \wy, \wz)) =  h \cdot X^{m+2} s^f d\tau \, ds \, du \, d\wz,
\end{align*}
where $h$ is a bounded distribution on $\mathscr{M}^2_h$. 
Hence, using \eqref{XH} we arrive after cancellations at the estimates ($j=1,2,3$)
\begin{equation*}
I_j \leq  \|\w\|_{\infty} \, \int s^{-2+f+\mu_0+\gamma+\A} h \, G_j \, d\tau \, ds \, du \, d\wz 
\leq C \, \|\w\|_{\infty}
\end{equation*}
for some uniform constant $C>0$. Summing up, we conclude
\begin{align}
|I_1+I_2+I_3| \leq C \|\w\|_{\infty}.
\end{align}

\subsubsection{Estimates near the lower right corner of the front face:}
Let us assume that the heat kernel $H$ is compactly supported near the lower right corner of the front face. 
Its asymptotic behaviour is appropriately described in the following projective coordinates
\begin{align*}
\tau=\frac{t-\wt}{\wx^2}, \ s=\frac{X}{\wx}, \ u=\frac{Y-\widetilde{y}}{\wx}, \ Z, \wx, \ Y, \ \widetilde{z},
\end{align*}
where in these coordinates $\tau, s, \wx$ are the defining functions of tf, rf and ff respectively. 
The coordinates are valid whenever $(\tau, s)$ are bounded as $(t-\wt,x,\wx)$ approach zero. 
For the transformation rule of the volume form we compute
\begin{align*}
\beta^*(d\wt \dv(\wx, \wy, \wz))=h \cdot \wx^{m+1} d\tau \, d\wx \, du \, d\wz,
\end{align*}
where $h$ is a bounded distribution on $\mathcal{M}^2_h$. 
Hence we obtain using 
\eqref{XH} and $\A < (\mu_0-\gamma)$ after cancellations ($j=1,2,3$)
\begin{equation*}
\begin{split}
I_j &\leq \|\w\|_{\infty} \int s^{\mu_0-\gamma-\A} \, \wx^{-1} h \, G_j \, d\tau \, d\wx \, du \, d\wz \\
&\leq  \|\w\|_{\infty} \, X^{\mu_0-\gamma-\A} \int 
\wx^{-1-(\mu_0-\gamma-\A)} \, G_j \, d\tau \, d\wx \, du \, d\wz \\
&\leq  C\, \|\w\|_{\infty} \, X^{\mu_0-\gamma-\A} \int_X^\infty
\wx^{-1-(\mu_0-\gamma-\A)} \, d\wx
\leq  C\, \|\w\|_{\infty},
\end{split}
\end{equation*} 
for some uniform constant $C>0$. Summing up, we conclude 
\begin{align}
|I_1+I_2+I_3| \leq C \, \|\w\|_{\A} 
\end{align}

\subsubsection{Estimates where the diagonal meets the front face:} 
We assume that the heat kernel $H$ is compactly supported near the intersection of the 
temporal diagonal td and the front face. Its asymptotic behaviour is conveniently described 
using the following projective coordinates
\begin{align}
\eta^2=\frac{t-\wt}{x^2}, \ S =\frac{(x-\wx)}{\sqrt{t-\wt \ }}, \ 
U= \frac{y-\wy}{\sqrt{t-\wt \ }}, \ Z =\frac{x (z-\wz)}{\sqrt{t-\wt \ }}, \  x, \ 
\wy, \ \wz.
\end{align}
In these coordinates tf is the face in the limit $|(S, U, Z)|\to \infty$, 
ff and td are defined by $x, \eta$, respectively. 
For the transformation rule of the volume form we compute
\begin{align*}
\beta^*(d\wt \dv(\wx, \wy, \wz))=h \cdot x^{m+2} \eta^{m+1} d\eta \, dS \, dU \, dZ,
\end{align*}
where $h$ is a bounded distribution on $\mathcal{M}^2_h$. 
Consequently, using Theorem \ref{heat-asymptotics} ($j=1,2,3$)
\begin{equation*}
\begin{split}
I_j \leq \int \eta^{-2} G_j(x, \eta, S, U, Z, \wy, \wz) \, \w(\wt, x,y,z)  \, d\eta \, dS \, dU \, dZ, 
\end{split}
\end{equation*}
where $G_j$ is uniformly bounded at the boundary faces of $\mathscr{M}^2_h$.
Since the heat kernel is integrated against a constant $\w(x,y,z)$, the singularity in $\eta$ can be 
cancelled using integration by parts near td, as in the estimate of $I_4$ in \cite[\S 3.1]{BV}. This 
leads to an estimate 
\begin{align*}
|I_1+I_2+I_3| &\leq C \|\w\|_{\infty}.
\end{align*} 

\subsection{H\"older differences in time}\label{proof2}
Consider in the previously set notation $v \in \ho(M\times [0,T], S_0)_0$ and the integral operator $G$. 
For any two fixed time points $t$ and $t'$ (we shall assume $t\geq t'$ without loss of generality), 
as well as any fixed space point $p\in M$ 
we will establish the following estimate
\begin{equation}\label{time-hoelder-difference}
\left\| G v (p,t) - G v (p,t')\right\|  \leq C \| v \|'_{\A,0} |t-t'|^{\frac{\A}{2}},
\end{equation}
for some uniform constant $C>0$. 
We model our estimates after a similar analysis in \cite{BV}. We write
\begin{align*}
&G v (p,t) - G v (p,t') \\
&= \int_0^t \int_{M} G(t-\wt, p,\widetilde{p}) (v(\wt, \widetilde{p})-v(\wt, p)) \, d\wt \, \textup{dvol}_g(\widetilde{p})\\
&- \int_0^{t'} \int_{M} G(t'-\wt, p,\widetilde{p}) (v(\wt, \widetilde{p})-v(\wt, p)) \, d\wt \, \textup{dvol}_g(\widetilde{p})\\
&+ \int_0^t \int_{M} G(t-\wt, p,\widetilde{p}) \, v(\wt, p) \, d\wt \, \textup{dvol}_g(\widetilde{p})\\
&- \int_0^{t'} \int_{M} G(t'-\wt, p,\widetilde{p}) \, v(\wt, p) \, d\wt \, \textup{dvol}_g(\widetilde{p}).
\end{align*}

Let us first assume $t\geq 2t'$. Then $t, t' \leq 2|t-t'|$ and we may estimate the first two
integrals exactly as $J'_1$ and $J'_2$ in \cite{BV}. For the last two integrals we note that the estimates 
at the boundary faces of $\mathscr{M}^2_h$ yield additional powers of $(\sqrt{t})^\A$ and $(\sqrt{t'})^\A$.
Using the fact that $t, t' \leq 2|t-t'|$, we obtain the estimate \eqref{time-hoelder-difference} as well.
Let us now assume $t'\leq t<2t'$. Note that then $(2t'-t)$ is smaller than $t$ and $t'$.
We introduce the following notation 
\begin{align*}
T_+:= [2t'-t, t], \quad T'_+:= [2t'-t, t'], \quad T_-:= [0,2t'-t]. 
\end{align*}
We can now decompose the integrals above accordingly and obtain
\begin{align*}
&G v (p,t) - G v (p,t') \\
&= \int_{T_-} \int_{M} (G(t-\wt, p,\widetilde{p}) - 
G(t'-\wt, p,\widetilde{p})) (v(\wt, \widetilde{p})-v(\wt, p)) \, d\wt \, \textup{dvol}_g(\widetilde{p})\\
&+ \int_{T_+} \int_{M} G(t-\wt, p,\widetilde{p}) (v(\wt, \widetilde{p})-v(\wt, p)) \, d\wt \, \textup{dvol}_g(\widetilde{p}) \\
&- \int_{T'_+} \int_{M} G(t'-\wt, p,\widetilde{p}) (v(\wt, \widetilde{p})-v(\wt, p)) \, d\wt \, \textup{dvol}_g(\widetilde{p})
\end{align*}
\begin{align*}
&+ \int_{T_-} \int_{M} (G(t-\wt, p,\widetilde{p}) - 
G(t'-\wt, p,\widetilde{p})) \, v(\wt, p) \, d\wt \, \textup{dvol}_g(\widetilde{p}) \\
&+ \int_{T_+} \int_{M} G(t-\wt, p,\widetilde{p}) \, v(\wt, p) \, d\wt \, \textup{dvol}_g(\widetilde{p})\\
&- \int_{T'_+} \int_{M} G(t'-\wt, p,\widetilde{p}) \, v(\wt, p) \, d\wt \, \textup{dvol}_g(\widetilde{p})
\\ &=: K_1 + K_2 + K_3 + K_4 + K_5 + K_6.
\end{align*}

Note that as in Remark \ref{why-such-spaces}, with 
$v \in \ho_{,0} = \ho \, \cap \, x^{\A} \mathcal{C}^0_{\textup{ie}}$, 
we use the H\"older regularity $v \in \ho$ in the estimates of $K_1, K_2, K_3$, 
and use an additional $x^\A$-weight in $v \in x^{\A} \mathcal{C}^0_{\textup{ie}}$ 
in the estimates of $K_4, K_5, K_6$.
\medskip

The first term $K_1$ is estimated exactly as the terms $J_3$ in \cite[\S 3.2]{BV}.
The second term $K_2$ is estimated exactly as the term $J_1$ in \cite[\S 3.2]{BV}.
The third term $K_3$ is estimated exactly as the term $J_2$ in \cite[\S 3.2]{BV}.
It remains to estimate the other terms $K_4, K_5, K_6$. Note that for $K_4$ and some
$\theta \in (t',t)$ we obtain with $\A = 1/\N$ as in Lemma \ref{mean-lemma}
\begin{align*}
K_4  = |t-t'|^\frac{\A}{2} \int_0^{2t'-t} \int_{M} (\theta - \wt)^{1-\frac{\A}{2}}
\partial_\tau (G) (\theta-\wt, p,\widetilde{p})
 \, v(\wt, p) \, d\wt \, \textup{dvol}_g(\widetilde{p}). 
 \end{align*}
We proceed in the notation of the previous subsection. The estimates 
use only the $x^\A \mathcal{C}^0_{\textup{ie}}$ bound of 
$v \equiv x^\A \w \in \ho(M\times [0,T], S)_{0} \subset x^\A \mathcal{C}^0_{\textup{ie}}$.
For the purpose of brevity, we omit the estimates at the top corner of ff 
and just point out that the estimates are parallel to those near the lower right corner
with same front face behaviour.

\subsubsection{Estimates near the lower left corner of the front face:}
Note that near the left lower corner of the front face, $x^2\geq (\theta-\wt)$. 
Consequently, $x^{-2}\leq (\theta-\wt)^{-1}$ and in particular for any $\delta > 0$ we find 
\begin{align}\label{delta-estimate}
\int x^{2-\delta} d\tau = \int  x^{-\delta} d\wt \leq \int (\theta-\wt)^{-\frac{\delta}{2}} d\wt. 
\end{align}
We compute after cancellations using \eqref{G-primed}
\begin{align*}
|K_4|   &\leq C |t-t'|^\frac{\A}{2} \|\w\|_\infty \int G_4 \, d\tau \, ds \, du \, d\wz
\leq C |t-t'|^\frac{\A}{2} \|\w\|_\infty, \\
|K_5|  &\leq C \|\w\|_\infty \int x^\A \, G_5 \, d\tau \, ds \, du \, d\wz
\leq C \|\w\|_\infty \int_{2t'-t}^t (\theta-\wt)^{-1+\frac{\A}{2}}  \, d\wt, \\
|K_6|  &\leq C \|\w\|_\infty \int x^\A \, G_6 \, d\tau \, ds \, du \, d\wz
\leq C \|\w\|_\infty \int_{2t'-t}^{t'} (\theta-\wt)^{-1+\frac{\A}{2}}  \, d\wt,
\end{align*}
where all kernels $G_j$ are bounded at the boundary faces of the heat space 
$\mathscr{M}^2_h$. From there we conclude using \eqref{C}
$$
|K_4|+|K_5|+|K_6| \leq C |t-t'|^{\frac{\A}{2}} \|\w\|_\infty.
$$

\subsubsection{Estimates near the lower right corner of the front face:}
We compute after cancellations
\begin{align*}
|K_4|  &\leq C |t-t'|^\frac{\A}{2} \|\w\|_\infty \int \wx^{-1} s^{\mu_0-\gamma} G_4 \, d\tau \, d\wx \, du \, d\wz \\
|K_5|  &\leq C \|\w\|_\infty \int \wx^{-1+\A} s^{\mu_0-\gamma} G_5 \, d\tau \, d\wx \, du \, d\wz \\
|K_6|  &\leq C \|\w\|_\infty \int \wx^{-1+\A} s^{\mu_0-\gamma} G_6 \, d\tau \, d\wx \, du \, d\wz ,
\end{align*}
where all kernels $G_j$ are bounded at the boundary faces of the heat space 
$\mathscr{M}^2_h$. Observe that near the lower right corner we may estimate
$$
\int \left| \wx^{-1-(\mu_0-1)} s^{\mu_0-1} G_j\right| d\wx \leq \textup{const}.
$$
Consequently, we obtain as in \eqref{delta-estimate}
\begin{align*}
|K_4|  &\leq C |t-t'|^\frac{\A}{2} \|\w\|_\infty, \\
|K_5|  &\leq C \|\w\|_\infty \int_{2t'-t}^t (\theta-\wt)^{-1+\frac{\A}{2}} \, d\wt, \\
|K_6|  &\leq C \|\w\|_\infty \int_{2t'-t}^{t'} (\theta-\wt)^{-1+\frac{\A}{2}} \, d\wt.
\end{align*}
From there we conclude using \eqref{C}
$$
|K_4|+|K_5|+|K_6| \leq C |t-t'|^{\frac{\A}{2}} \|\w\|_\infty. 
$$

\subsubsection{Estimates where the diagonal meets the front face:} 
We compute after cancellations 
\begin{align*}
|K_4|  &\leq C |t-t'|^\frac{\A}{2} \int \eta^{-1-\A} G_4(x, \eta, S, U, Z, \wy, \wz) \, \w(\wt, x,y,z)  \, d\eta \, dS \, dU \, dZ, \\
|K_5|  &\leq C \int x^\A \, \eta^{-1} G_5(x, \eta, S, U, Z, \wy, \wz) \, \w(\wt, x,y,z)  \, d\eta \, dS \, dU \, dZ, \\
|K_6|  &\leq C \int x^\A \, \eta^{-1} G_6(x, \eta, S, U, Z, \wy, \wz) \, \w(\wt, x,y,z)  \, d\eta \, dS \, dU \, dZ.
\end{align*}
where all kernels $G_j$ are bounded at the boundary faces of the heat space 
$\mathscr{M}^2_h$. Since the heat kernel is integrated against a constant $\w(x,y,z)$, the singularity in $\eta$ can be 
cancelled using integration by parts near td, as in the estimate of $I_4$ in \cite[\S 3.1]{BV}. This 
leads to an estimate 
\begin{align*}
|K_4|  &\leq C |t-t'|^\frac{\A}{2} \|\w\|_\infty \int \eta^{-\A} G'_4(x, \eta, S, U, Z, \wy, \wz) \, d\eta \, dS \, dU \, dZ, \\
|K_5|  &\leq C \|\w\|_\infty \int x^\A \, G'_5(x, \eta, S, U, Z, \wy, \wz) \, d\eta \, dS \, dU \, dZ, \\
|K_6|  &\leq C \|\w\|_\infty \int x^\A \, G'_6(x, \eta, S, U, Z, \wy, \wz) \, d\eta \, dS \, dU \, dZ,
\end{align*}
where all kernels $G'_j$ are still bounded at the boundary faces of the heat space 
$\mathscr{M}^2_h$. The estimates now follow along the lines of the estimates of $J_1, J_2$
and $J_3$ in \cite[\S 3.2]{BV} near td.

\subsection{Estimates of the supremum}\label{proof3}

Consider as before $\w \in \mathcal{C}^0_{\textup{ie}}(M\times [0,T], S_0)$.  
In this subsection we estimate the supremum norm of the following integral
$$
J:= \int_0^t \int_{M} G'(t-\wt, p,\widetilde{p}) \w(\wt, \widetilde{p}) \, d\wt \, \textup{dvol}_g(\widetilde{p}),
$$
where $G'=X^{-\gamma} \circ \V^2 e^{-t\Delta_L} \circ X^{-2+\gamma+\A}$. As before, we assume that the kernel $G'$ is 
compactly supported near the various corners of the front face in the heat space blowup
$\mathscr{M}^2_h$, where for convenience we write out the corresponding projective coordinates 
once again. The estimates are classical away from the front face and hence we may assume 
that $p=(x,y,z)\in \U$. Moreover, as before it suffices to integrate over the singular neighborhood 
$\U$ with $\widetilde{p} = (\w, \wy, \wz)$, replacing the integration region $M$ in the integral $J$ by $\U$.

\subsubsection{Estimates near the lower left corner of the front face:}
Assume that the integral kernel $G'$ is compactly supported near the lower left corner of the front face. 
We employ as before the following projective coordinates
\begin{align*}
\tau=\frac{t-\wt}{x^2}, \ s=\frac{\wx}{x}, \ u=\frac{y-\widetilde{y}}{x}, \ x, \ y, \ z, \ \widetilde{z}.
\end{align*}
where in these coordinates $\tau, s, x$ are the defining functions of tf, lf and ff respectively. 
For the transformation rule of the volume form we compute
\begin{align*}
\beta^*(d\wt \dv(\wx, \wy, \wz)) =  h \cdot x^{m+2} s^f d\tau \, ds \, du \, d\wz,
\end{align*}
where $h$ is a bounded distribution on $\mathscr{M}^2_h$. 
Hence, using \eqref{XH} we arrive for any $\w\in \ho$ after cancellations at the estimates
\begin{equation*}
\begin{split}
|J| \leq \|\w\|_\infty \int x^\A G''(s, \tau, u, x, y, z, \wz) d\tau \, ds \, du \, d\wz \leq C \, x^\A \, \|\w\|_\infty,
\end{split}
\end{equation*}
for some uniform constant $C>0$ and bounded function $G''$ on $\mathscr{M}^2_h$.

\subsubsection{Estimates near the lower right corner of the front face:}
Assume that the heat kernel $H$ is compactly supported near the lower right corner of the front face. 
We employ as before the following projective coordinates
\begin{align*}
\tau=\frac{t-\wt}{\wx^2}, \ s=\frac{x}{\wx}, \ u=\frac{y-\widetilde{y}}{\wx}, \ \wx, \ y, \ z, \ \widetilde{z}.
\end{align*}
where in these coordinates $\tau, s, x$ are the defining functions of tf, rf and ff respectively. 
For the transformation rule of the volume form we compute
\begin{align*}
\beta^*(d\wt \dv(\wx, \wy, \wz)) =  h \cdot \wx^{m+1} d\tau \, d\wx \, du \, d\wz,
\end{align*}
where $h$ is a bounded distribution on $\mathscr{M}^2_h$. 
Hence, using \eqref{XH} and the fact that $x\leq \wx$ near the lower right corner, 
we arrive for any $\w\in \ho$ after cancellations at the estimates
\begin{equation*}
\begin{split}
|J| &\leq \|\w\|_\infty \int \wx^{-1+\A} s^{\mu_0-\gamma} G''(s, \tau, u, \wx, y, z, \wz) \, d\tau \, d\wx \, du \, d\wz \\
& \leq C \, \|\w\|_\infty \ x^{\mu_0-\gamma} \int_x^\infty \wx^{-1-(\mu_0-\gamma)+\A} d\wx \leq C \, x^\A \, \|\w\|_\infty,
\end{split}
\end{equation*}
for some uniform constant $C>0$ and bounded function $G''$ on $\mathscr{M}^2_h$. 
Note that we used $\A < (\mu_0-\gamma)$ in the estimate above. 

\subsubsection{Estimates near the top corner of the front face:}
Assume that the heat kernel $H$ is compactly supported near the top corner of the front face. 
We employ as before the following projective coordinates
\begin{align*}
\rho=\sqrt{t-\wt}, \  \xi=\frac{x}{\rho}, \ \widetilde{\xi}=\frac{\wx}{\rho}, \ u=\frac{y-\wy}{\rho}, \ y, \ z, \ \wz,
\end{align*}
where in these coordinates, $\rho, \xi, \widetilde{\xi}$ are the defining 
functions of the boundary faces ff, rf and lf respectively. 
For the transformation rule of the volume form we compute
\begin{align*}
\beta^*(d\wt \dv(\wx, \wy, \wz)) =  h \cdot \rho^{m+1} \widetilde{\xi}^f \, d\rho \, d\widetilde{\xi} \, du \, d\wz,
\end{align*}
where $h$ is a bounded distribution on $\mathscr{M}^2_h$. 
Hence, using \eqref{XH} and the fact that $x\leq \rho$ near the lower right corner, 
we arrive for any $\w\in \ho$ after cancellations at the estimates
\begin{equation*}
\begin{split}
|J| &\leq \|\w\|_\infty \int \rho^{-1+\A} \xi^{\mu_0-\gamma} G''(\rho, \xi, \widetilde{\xi}, u, y, z, \wz) \, 
d\rho \, d\widetilde{\xi} \, du \, d\wz \\
& \leq C \, \|\w\|_\infty \ x^{\mu_0-\gamma} \int_x^\infty \rho^{-1-(\mu_0-\gamma)+\A} \, 
d\rho \leq C \, x^\A \, \|\w\|_\infty,
\end{split}
\end{equation*}
for some uniform constant $C>0$ and bounded function $G''$ on $\mathscr{M}^2_h$. 
Note that we used $\A < (\mu_0-\gamma)$ in the estimate above. 

\subsubsection{Estimates where the diagonal meets the front face:}\label{td-parts}
Assume that the heat kernel $H$ is compactly supported where the temporal diagonal meets the front face. 
Before we begin with the estimate, let us rewrite $J$ in following way
\begin{align*}
J &= \int_0^t \int_\U G'(t-\wt, x, y, z, \wx, \wy, \wz) (\w(\wt, \wx, \wy, \wz) - \w (t, x, y, z)) d\wt \dv(\wx, \wy, \wz)
\\&+ \int_0^t \int_\U G'(t-\wt, x, y, z, \wx, \wy, \wz) \w (t, x, y, z) d\wt \dv(\wx, \wy, \wz) =: J_1 + J_2.
\end{align*}
We employ as before the following projective coordinates
\begin{align*}
\eta^2=\frac{t-\wt}{x^2}, \ S =\frac{(x-\wx)}{\sqrt{t-\wt \ }}, \ 
U= \frac{y-\wy}{\sqrt{t-\wt \ }}, \ Z =\frac{x (z-\wz)}{\sqrt{t-\wt \ }}, \  x, \ 
y, \ z.
\end{align*}
where in these coordinates tf is the face in the limit $|(S, U, Z)|\to \infty$, 
ff and td are defined by $x, \eta$, respectively. 
For the transformation rule of the volume form we compute
\begin{align*}
\beta^*(d\wt \dv(\wx, \wy, \wz))=h \cdot x^{m+2} \eta^{m+1} d\eta \, dS \, dU \, dZ,
\end{align*}
where $h$ is a bounded distribution on $\mathscr{M}^2_h$. Note that in these coordinates
$$
d_M((x, y, z), (\wx, \wy, \wz)) = x \eta \sqrt{|S|^2 + |U|^2 + (2-\eta S) |Z|^2}.
$$
Hence, using \eqref{XH} we arrive for any $\w\in \ho$ after cancellations at the estimates
\begin{equation*}
|J_1| \leq \|\w\|_\A \int x^{2\A} \eta^{-1+\A}  G''(x, y, z, \eta, S, U, Z) \, d\eta \, dS \, dU \, dZ 
\leq C \, x^{2\A} \, \|\w\|_\A, 
\end{equation*}
for some uniform constant $C>0$ and bounded function $G''$ on $\mathscr{M}^2_h$. 
Estimating similarly for $J_2$ leads to a singular $\eta^{-1}$ behaviour at td, due to derivatives
of the form $\eta^{-1} \partial_S, \eta^{-1}\partial_U$ and $\eta^{-1}\partial_Z$. Due to the fact
that $J_2$ is comprised of the heat kernel integrated against $\w(t, x, y, z)$ which does not 
depend on $(S, U, Z)$, we obtain after integrating by parts for some bounded function 
$G''$ on $\mathscr{M}^2_h$ (assume e.g. $X=\eta^{-2}\partial_Z^2$)
\begin{equation*}
|J_2|  \leq \int x^\A \, G'' \, \w(t, x, y, z) \, \partial_z h(x - x \eta S, y- x\eta U, z- \eta Z) \, d\eta \, dS \, dU \, dZ|
\leq C \, x^\A \, \|\w\|_\infty.
\end{equation*}
\end{proof}

We conclude the section with stating the mapping properties for the Laplace Beltrami operator $\Delta$
acting on smooth functions over $M$. We identify $\Delta$ with its Friedrichs self-adjoint extension. 
Under stronger assumptions other than admissibility of the edge metric,
mapping properties of the heat operator have been established in our joint work with 
Bahuaud \cite[Theorem 3.2]{BV}. Here, following the arguments of the previous Theorem 
\ref{main-est1} one easily proves the following result.

\begin{thm}\label{main-est3} Consider an edge manifold $(M,g)$ with an admissible 
edge metric $g$ satisfying tangential stability as in Definition \ref{assump-spectrum}. 
Consider the index set at the right and left face of the heat kernel lifted to $\mathscr{M}^2_h$, 
with the minimal element $\mu_1> 0$. Fix any $\gamma \in (1-\dim F, \mu_1)$. 
Then for $\A \in (0,1)\cap (0,\mu_1-\gamma)$ the heat operator $e^{-t\Delta}$ 
for the Friedrichs self-adjoint extension of the Laplace Beltrami operator $\Delta$ 
defines a bounded mapping 
\begin{align*}
&e^{-t\Delta}: x^{-2+\gamma} \hok (M\times [0,T]) \to \mathcal{C}^{k+2,\A}_{\textup{ie}}(M\times [0,T])_\gamma, \\
&e^{-t\Delta}: x^{-2+\gamma+\varepsilon} \, \mathcal{C}^{k+1,\A}_{\textup{ie}}(M\times [0,T]) \to 
t^{\frac{\varepsilon}{2}} \mathcal{C}^{k+2,\A}_{\textup{ie}}(M\times [0,T])_{\gamma}.
\end{align*}
\end{thm}

The proof proceed along the lines of Theorem \ref{main-est1}.
We point out that due to stochastic completeness of the Laplace Beltrami heat operator,
one can completely avoid terms of the form $L_3$, compare 
\cite{BV} for the estimate of the H\"older differences. This allows us to use 
$\mathcal{C}^{k,\A}_{\textup{ie}}(M\times [0,T])_\gamma$ spaces of scalar functions 
which are defined without requiring better $x$-weight for the supremum norm, in contrast to the
H\"older space of sections of $S_0$. \medskip

Another crucial difference to Theorem \ref{main-est1} is that the higher order asymptotics of solutions in the target space 
$\mathcal{C}^{k+2,\A}_{\textup{ie}}(M\times [0,T])_\gamma$ arises only after differentiation. The reason is the 
$a(t,y) \rho_\rf^0$ leading order term in the asymptotics of the heat kernel $e^{-t\Delta}$ at the right face, 
which is independent of $(x,z)$ 
and hence vanishes under differentiation by $(x\partial_x)$ and $\partial_z$, but not under 
$(x\partial_y)$ and $x^2\partial_t$. This explains the peculiar definition of the H\"older space 
$\mathcal{C}^{k+2,\A}_{\textup{ie}}(M\times [0,T])_\gamma$ for scalar functions, 
which distinguishes the weights depending on the derivatives applied.
Apart from that, the estimates follow along the lines of the 
corresponding argument for the Lichnerowicz Laplacian.

\section{Short time existence of the Ricci de Turck flow}\label{short-section}

We proceed with the explicit analysis of the Ricci flow of an admissible $(\A, \gamma, k)$-H\"older 
regular incomplete edge metric $g$, satisfying tangential stability introduced in Definition \ref{assump-spectrum}. 
A particular consequence of the diffeomorphism invariance of the Ricci tensor is the
well-known fact that the Ricci flow is not a parabolic system. This analytic difficulty
is overcome using the standard de Turck trick with the background metric chosen 
as the initial incomplete edge metric $g \in \textup{Sym}^2({}^{ie}T^*M)$. \medskip

Writing the flow metric as
$(g+v)$ with $v \in \textup{Sym}^2({}^{ie}T^*M)$ and $v(0)=0$, we can follow the linearization
of the Ricci de Turck flow as e.g. in Bahuaud \cite[4.2]{Bah} and obtain a quasilinear parabolic system
for $v \in \textup{Sym}^2({}^{ie}T^*M)$, where all indices refer to the metric and curvature terms
as tensors on ${}^{ie}T^*M$. Let $Ric(g)$ and $R(g)$ denote the Ricci and Riemannian $(4,0)$
curvature tensors, respectively. Then the Ricci de Turck flow can be written as

\begin{equation}
\begin{split}
(\partial_t + \Delta_L) v_{ij} &= (T_1v)_{ij} + (T_2v)_{ij}  +(T_3v)_{ij}, \\
 (T_1v)_{ij} &= ((g+v)^{kl} - g^{kl}) (\nabla_k \nabla_l v)_{ij}, \\
  (T_3v)_{ij} &= (g+v)^{-1} * (g+v)^{-1} * \nabla v * \nabla v, \\
 (T_2v)_{ij} &= -2\textup{Ric}(g)_{ij} + Q((g+v)^{kl} (g+v)_{ip} g^{pq} R(g)_{jkql}) \\ 
 &\quad + Q((g+v)^{kl} (g+v)_{jp} g^{pq} R(g)_{ikql}), 
\end{split}
\end{equation}
where $Q(*)$ is obtained by taking a linear formal expansion of $(*)$ in $v$ and picking those
terms that are at least quadratic in $v$. Moreover, $\Delta_L$ and $\nabla$ denote the Lichnerowicz Laplacian and the Levi Civita
covariant derivative, respectively, both defined with respect to the initial metric $g$ and acting
on $\textup{Sym}^2({}^{ie}T^*M)$. \medskip

We decompose $v= ug \oplus \w$ into pure trace and trace-free parts with respect to the initial metric $g$. 
The Lichnerowicz Laplacian $\Delta_L$ respects the decomposition since 
$\textup{tr}_{g}(\Delta_L v) = \Delta(\textup{tr}_{g} (v))$ and $\Delta (ug) = (\Delta u) g$, 
where $\Delta$ on the right hand side of the latter equation is the Laplace Beltrami 
operator of $g$ acting on functions. \medskip

We also note the following useful expansion as in \cite[(4.1)]{Bah}
\begin{align*}
(g+v)^{ab} &\equiv ((1+u)g+ \w)^{ab} = 
\frac{g^{ab}}{(1+u)} - \frac{g^{al}g^{bm}}{(1+u)^2}\w_{ml}
\\ &+ \frac{((1+u)g+ \w)^{bl}g^{am}g^{pq}}{(1+u)^2}\w_{lp} \w_{mq}.
\end{align*}
Plugging this expansion into $T_1(v)$ we find 
\begin{align*}
T_1(v) &= 
\left( \frac{-u }{(1+u)} \Delta u - \frac{g^{al}g^{bm}}{(1+u)^2}\w_{ml} \partial_a\partial_b u
\right.  \\ &+ \left. \frac{((1+u)g+ \w)^{bl}g^{am}g^{pq}}{(1+u)^2}\w_{lp} \w_{mq}\partial_a\partial_b u\right) g \\
&+ \frac{-u }{(1+u)} \Delta \w - \frac{g^{al}g^{bm}}{(1+u)^2}\w_{ml} \nabla_a\nabla_b \w
\\ &+ \frac{((1+u)g+ \w)^{bl}g^{am}g^{pq}}{(1+u)^2}\w_{lp} \w_{mq}\nabla_a\nabla_b \w
\end{align*}
Let us study the singular structure of $T_1(v)$. 
Note that if the lower index $a$ refers to the radial coordinate $x$ 
or to the edge coordinates $y$, then $\nabla_a$
acts on $S_0=\textup{Sym}^2_0({}^{ie}T^*M)$ as a combination of derivatives $x^{-1}\V$
and $x^{-1}$ times a smooth function on $\overline{M}$, smooth up to 
the boundary. If the lower index $a$ refers to tangential coordinates $z$, then 
$\nabla_a$ acts on $S_0$ as a combination of derivatives $\V$
and smooth functions on $\overline{M}$. On the other hand, any upper index $a$ referring 
to the radial coordinate $x$ or the edge coordinates $y$, contributes no singular $x$ factor due to the structure of the 
inverse metric $g^{-1}$, while an upper index $a$ referring to the tangential coordinates $z$
contributes a factor $x^{-1}$. Counting the factors, we conclude
\begin{align*}
T_1(v) &= 
\left( \frac{-u }{(1+u)} \Delta u + \frac{1}{x^2} O_1(\w) O_1(\V^2 u)  \right)  g \\
&\quad + \frac{-u }{(1+u)} \Delta_L \w + \frac{1}{x^2} O_2(\w, \V \w, \V^2 \w),
\end{align*}
where $O_1(*)$ and $O_2(*)$ refers to any at least linear and at least quadratic combination
of the term $(*)$ in the brackets, respectively. In each of the summands we do not indicate notationally 
further factors which include just bounded combinations of smooth (up to the boundary) functions, $u$ and $\w$, with
at most edge $\VV$ derivatives. Counting singular $x^{-1}$-factors as before we obtain
\begin{align*}
T_2(v) &=  -2\textup{Ric}(g) + \frac{1}{x^2} O_1(\w)O_1(u) + \frac{1}{x^2} O_2(\w), \\
T_3(v) &= \frac{1}{x^2} O_2(\V u) + \frac{1}{x^2} O_1(\w, \V \w) O_1(\V u) + \frac{1}{x^2} O_2(\w, \V \w). 
\end{align*}
where in case of $T_2(v)$ we used the fact that components of the Riemannian curvature 
$(4,0)$ tensor of an edge metric of H\"older regular geometry are $O(x^{-2})$ as $x\to 0$, when acting
on ${}^{ie}TM$. We point out that $T_2(v)$ does not admit terms of the form $x^{-2}O_2(u)$ due to 
cancellations. \medskip

We decompose the Ricci curvature into the trace free component $\frac{\textup{scal}(g)}{m} g$
and the trace-free part of the Ricci curvature tensor $\textup{Ric}'(g)$. Summarizing our 
analysis from above we now obtain under the direct sum decomposition into pure
trace and trace free components (with respect to the initial edge metric $g$) the following
structure of the Ricci de Turck flow

\begin{equation}\label{ricci-de-turck-linearization}
\begin{split}
\left( \partial_t  + \Delta \oplus \Delta_L \right) (u \oplus \w)  &= \left(\left(-\frac{u}{1+u} \Delta u + 
\frac{1}{x^2} O_1(\w) O_1(\V^2 u) + \frac{\textup{scal}(g)}{m}\right)\right. \\
&\oplus \left.\left(-\frac{u }{(1+u)} \Delta_L \w - 2\textup{Ric}'(g)\right)\right) + \frac{1}{x^2} O_2(\w, \V \w, \V^2 \w)
\\ &+ \frac{1}{x^2} O_1(\w, \V \w) O_1(u, \V u) + \frac{1}{x^2}O_2(\V u) =: F(u, \w).
\end{split}
\end{equation}

In order to set up a fixed point argument for that non-linear equation, we follow the outline of 
\cite[Theorem 4.1]{BV} and introduce the following Banach space for any $\gamma_0,\gamma_1>0$ 
and $\alpha \in (0,1)$
\begin{align}
H_{\gamma_0, \gamma_1} := \mathcal{C}^{k+2,\A}_{\textup{ie}}(M\times [0,T])_{\gamma_1} \oplus
\mathcal{C}^{k+2,\A}_{\textup{ie}}(M\times [0,T], S_0)_{\gamma_0}.
\end{align}

We can always choose $\gamma_0, \gamma_1, \alpha > 0$ sufficiently small such that the 
following algebraic relations are satisfied

\begin{equation}\label{gamma-relations}
\begin{split}
&\textup{(i)} \quad \gamma_0 \in (0, \mu_0), \quad \gamma_0 \leq 2\gamma_1, \quad \gamma_0 < \gamma, \\
&\textup{(ii)} \quad \gamma_1 \in (0, \mu_1), \quad \gamma_1 \leq \gamma_0, \quad \ \ \gamma_1 < \gamma, \\
&\textup{if $dim B > 0$, then} \, \gamma_0 \leq 2 \min\{1, \gamma_1\}, \quad \gamma_1 \leq 2, \\
&\textup{(iii)} \quad \A \in (0,(\mu_0 - \gamma_0)) \cap (0, (\mu_1-\gamma_1)).
\end{split}
\end{equation}

\begin{thm}\label{existence1}
Consider an admissible $(\A, \gamma, k+1)$-H\"older regular edge manifold $(M,g)$ with an edge singularity at $B$,
satisfying tangential stability introduced in Definition \ref{assump-spectrum} with minimal elements
$\mu_0, \mu_1>0$ of the index sets at the right and left faces. Then the Riemannian 
metric $g$ may be evolved under the Ricci de Turck flow as\footnote{The decomposition 
$g(t) = (1+u) \oplus \w$ into pure trace and trace-free components is
with respect to $g(0)=g$.} $g(t) = (1+u) \oplus \w$
within the Banach space $H_{\gamma_0, \gamma_1}$ on some finite time interval $[0,T]$,
where $\gamma_0, \gamma_1, \alpha > 0$ are sufficiently small and satisfy \eqref{gamma-relations}.
\end{thm}

\begin{proof}
Consider first the linearization of the Ricci de Turck flow in \eqref{ricci-de-turck-linearization}.
Consider $(u, \w) \in H_{\gamma_0, \gamma_1}$. Then, in view of the Definition \ref{funny-spaces},
the regularity of the individual terms in 
the expression for $F(u, \w)$ is as follows (according to the ordering of terms in the expression 
\eqref{ricci-de-turck-linearization})
\begin{equation}\label{F-regularity}
\begin{split}
F(u, \w) &\in \left(x^{-2+\min\{2, \gamma_1\}} \hok + x^{-2+\gamma_0 + \min\{2, \gamma_1\}} \hok + 
x^{-2+\gamma} \mathcal{C}^{k+1,\A}_{\textup{ie}}\right) \\
& \oplus \left(x^{-2+\gamma_0} \hok + x^{-2+\gamma} \mathcal{C}^{k+1,\A}_{\textup{ie}} \right) 
+ x^{-2+2\gamma_0} \hok \\ &+ x^{-2+\gamma_0} \hok + x^{-2+2 \min\{1, \gamma_1\}} \hok.
\end{split}
\end{equation}
In case of $\dim B = 0$, there are no $x\partial_y$ derivatives and 
we may replace $\min\{2, \gamma_1\}$ and $\min\{1, \gamma_1\}$ by $\gamma_1$
in \eqref{F-regularity}. Using the algebraic relations \eqref{gamma-relations} we conclude that 
\begin{align*}
F \left( H_{\gamma_0, \gamma_1} \right) \subseteq 
\left(x^{-2+\gamma_1}\hok (M\times [0,T])\right) \oplus \hok (M\times [0,T],S_0)_{-2+\gamma_0}.
\end{align*}
Using the mapping properties of Theorems \ref{main-est1} and \ref{main-est3}, we find
\begin{align*}
\Phi := \left(e^{-t\Delta} \oplus e^{-t\Delta_L}\right) \circ F: H_{\gamma_0, \gamma_1} \to H_{\gamma_0, \gamma_1}.
\end{align*}
Solution to the Ricci de Turck flow is by construction a fixed point of $\Phi$. In order 
to prove existence of such a fixed point, we restrict $\Phi$ to a subset of 
$H_{\gamma_0, \gamma_1}$ and define
\begin{align*}
Z_\mu := \{(u, \w) \in H_{\gamma_0, \gamma_1} \mid \|(u, \w)\|_{H_{\gamma_0, \gamma_1}} \leq \mu\}, \quad \mu >0.
\end{align*}
The terms in the linearization \eqref{ricci-de-turck-linearization} are either quadratic in $(u, \w)$
or constant given by the summands $\textup{scal}(g)$ and $\textup{Ric}'(g)$ depending only on the initial metric.
Using the second mapping properties in Theorems \ref{main-est1} and \ref{main-est3}, the $H_{\gamma_0, \gamma_1}$ norm 
of $e^{-t\Delta} \textup{scal}(g) \oplus e^{-t\Delta_L}\textup{Ric}'(g)$ can be made smaller than $\mu/2$ if $T>0$ is sufficiently 
small. Since the other terms in $F(u,\w)$ are quadratic in $(u,\w)$, we find that $\Phi$ maps $Z_\mu$ to itself for $T>0$ and 
$\mu>0$ sufficiently small. Moreover, for $\mu>0$ sufficiently small, $\Phi$ satisfies the contraction mapping property 
\begin{align*}
\|\Phi(u, \w) - \Phi(u', \w')\|_{H_{\gamma_0, \gamma_1}} \leq q \|(u, \w)-(u', \w')\|_{H_{\gamma_0, \gamma_1}}
\end{align*}
with some positive $q<1$ for all $(u, \w)$ and $(u', \w') \in Z_\mu$. 
Hence, repeating the argument of \cite[Theorem 4.1]{BV} 
verbatim, the fixed point exists in $Z_\mu \subset H_{\gamma_0, \gamma_1}$.
\end{proof}

\section{Singular edge structure of the Ricci de Turck flow}\label{edge-flow} 

In this section we explain in what sense the evolved Ricci de Turck metric $g(t)$ remains 
an admissible incomplete edge metric. Recall $g(t) = (1+u)g + \w$, where $g$ is the initial admissible edge metric, 
$u \in \mathcal{C}^{k+2,\A}_{\textup{ie}} (M\times[0,T])_{\gamma_1}$ and $\w \in \mathcal{C}^{k+2,\A}_{\textup{ie}}(M\times[0,T], S_0)_{\gamma_0}$ is 
a higher order trace-free (with respect to $g$) term. Consider first how the conformal transformation
of $g$ into $(1+u)g$ affects the incomplete edge structure of the metric. The argument is worked out in 
\cite{BV2} as well. \medskip

Choose local coordinates $(x,y,z)$ near the 
singularity as before. Due to the fact that an element of $\ho$ must be independent of $z$ at $x=0$, 
we may write $u_0(y):= u(0,y,z)$. Since $u \in \mathcal{C}^{k+2,\A}_{\textup{ie},\gamma_1}$ 
we may apply the mean value theorem and find as in
Corollary \ref{mean-cor} that $x^{-\gamma_1}(u(x,y,z)-u_0(y)) = \xi^{-\gamma_1} (\xi \partial_\xi) u(\xi, y, z)$ with 
$\xi \in (0,x)$, and hence is bounded up to the edge singularity. Consequently 
we obtain a partial asymptotic expansion of $u$ as $x\to 0$
\[ u(x,y,z) = u_0(y) + O(x^{\gamma_1}). \]

Now we substitute $\xt = (1+u_0)^{\frac{1}{2}} x$. For small $u_0$ this defines a new boundary defining function, which 
varies along the edge. Consider the leading order term $\overline{g}$ of $g$, which is given 
by $\overline{g} \!\mid_{\U} = dx^2 + x^2 g_F + \phi^* g_B$ over the singular neighborhood $\U$.
We compute
\begin{equation}
\begin{split}
&(1+u) (dx^2 + \phi^*g^B +  x^2 g^F) \\ =& \ (1+u_0) dx^2 + 
\phi^* ((1+u_0) g^B) + (1+u_0) x^2 g^F + O(x^{\gamma_1})\\ 
=& \ d\xt^2 + \phi^*( (1+u_0) g^B ) + \xt^2 g^F + O(x^{\gamma_1}).
\end{split}
\end{equation}
The key point here is that up to a conformal transformation of the base metric on $B$, 
the leading term of the metric has the same rigid edge structure in the new choice of a boundary
defining function $\wx$. The trace-free term $\w$ is of higher order $O(x^{\gamma_0})$ as $x\to 0$. 
Consequently, up to a change of a boundary defining function and up to higher order terms, $g(t)$ is again an admissible
edge metric in the sense of Definition \ref{admissible}, extended to allow for the metric along the edge 
to be only H\"older regular and not necessarily smooth, and to include higher order terms $h$
with $|h|_g = o(1)$ as $x \to 0$ that are only H\"older regular but not necessarily smooth.

\section{Passing from the Ricci de Turck to the Ricci flow}\label{Ricci-flow}

The solution $g(t)$ of the Ricci de Turck flow is related to the actual Ricci flow by a diffeomorphism, 
a meanwhile classical trick of de Turck which we now make explicit, cf. \cite{Chow}. We employ the 
Einstein notational convention for summation of indices and define the time-dependent
de Turck vector field $W(t)$, given in a choice of local coordinates by the following expression
\begin{align*}
W(t)^j := g^{pq}(t) \left(\Gamma^j_{pq} (g(t)) - \Gamma^j_{pq} (g)\right),
\end{align*} 
where $\Gamma^j_{pq} (g(t))$ and $\Gamma^j_{pq} (g)$ denote the Christoffel symbols of the 
Ricci de Turck flow metric $g(t)$ and the initial admissible edge metric $g$, respectively. The Christoffel
symbols are not coordinate invariant and are given in the fixed choice of local coordinates by 
\begin{align*}
\Gamma^j_{pq} (g) = \frac{1}{2} g^{jm} \left( \partial_p g_{mq} + \partial_q g_{mp} - \partial_m g_{pq}\right),
\end{align*}
with $\Gamma^j_{pq} (g)$ obviously defined by the same expression with $g$ replaced by $g(t)$. From the expressions above
it is clear that the de Turck vector field $W(t)$ is a linear combination of vector fields $x^{-1}\V$ with 
$x^{-1+\overline{\gamma}} \mathcal{C}^{k+1,\A}_{\textup{ie}}$ regular coefficients; 
where 
\begin{equation}\label{gamma-overline}
\begin{split}
&\overline{\gamma}= \min\{\gamma_0, \gamma_1\}, \quad \textup{if} \ \dim B = 0, \\
&\overline{\gamma}=\min\{\gamma_0, \gamma_1, 1\}, \quad \textup{if} \ \dim B > 0,
\end{split}
\end{equation} 
due to possible $\partial_y$ derivatives.
\medskip

The de Turck vector field defines the corresponding one-parameter family of diffeomorphisms $\phi(t):M\to M$,
with $x^{-1+\overline{\gamma}} \mathcal{C}^{k+1,\A}_{\textup{ie}}$ regular 
components with respect to the local coordinates $(x,y,z)$ near the edge.
However, a priori we do not have a uniform existence time for $\phi(t)$ the closer we get to the singularity.
This is due to the fact that the $\partial_x$ component of the de Turck vector field need not be inward pointing
at $x=0$, unless we require that $\overline{\gamma} > 1$. In view of \eqref{gamma-overline}, 
$\overline{\gamma} > 1$ can only be satisfied in case of conical singularities $\dim B = 0$. \medskip

Assuming for the moment that $\phi(t)$ exists for a short time uniformly up to the edge singularity, 
we obtain a solution $g'(t)$ to the Ricci flow by setting
$g'(t):= \phi(t)^*g(t) = g(d\phi [\cdot], d\phi [\cdot])$. Due to additional derivatives, we conclude
\begin{align}
g' \in x^{-2+\overline{\gamma}} \mathcal{C}^{k,\A}_{\textup{ie}}
(M\times [0,T], \textup{Sym}^2({}^{ie}T^*M)).
\end{align}
This proves the following short time existence statement.
  
\begin{thm}
Consider an admissible $(\A, \gamma, k+1)$-H\"older regular edge manifold $(M,g)$,
satisfying tangential stability with minimal elements $(\mu_0, \mu_1)$. 
Assume that the de Turck vector field is inward pointing at $x=0$. 
This is true e.g. if $\dim B = 0$ and the minimal elements $\mu_0, \mu_1 > 1$, so that 
we may choose $\gamma_0, \gamma_1>1$ subject to the algebraic
relations \eqref{gamma-relations} and consequently $\overline{\gamma} 
= \min\{\gamma_0, \gamma_1\} > 1$. \medskip

Then the Riemannian metric $g$ 
may be evolved under the Ricci flow with 
\begin{equation}
g'(t) \in x^{-2+\overline{\gamma}} \mathcal{C}^{k,\A}_{\textup{ie}}
(M\times [0,T], S)
\end{equation} 
on some finite time interval $t\in [0,T]$. If $\mu_0, \mu_1 > 2$ so that we may choose 
$\gamma_0, \gamma_1 \geq 2$, then $g' \in \hok$
acts boundedly on $x^{-1} \V$ vector fields and is in that sense an edge metric.
\end{thm}

\section{Evolution of the Riemannian curvature tensor along the flow}\label{Ricci-bounded-section}

In this section we prove that the Riemannian curvature tensor of the Ricci flow metric $g'(t)$
is bounded along the flow for $t\in (0,T]$ when starting at an admissible H\"older regular edge manifold $(M,g)$
with bounded Riemannian curvature. More precisely we prove the following theorem.

\begin{thm}
Consider an admissible $(\A, \gamma, k+1)$-H\"older regular edge manifold $(M,g)$
satisfying tangential stability. Consider the Ricci de Turck flow
solution $g(t)=(1+u)g + \w$, where $\w$ trace-free with respect to $g$ and
\begin{align*}
(u, \w) \in H_{\gamma_0, \gamma_1} = \mathcal{C}^{k+2,\A}_{\textup{ie}}(M\times [0,T])_{\gamma_1} \oplus
\mathcal{C}^{k+2,\A}_{\textup{ie}}(M\times [0,T], S_0)_{\gamma_0},
\end{align*}
subject to the algebraic relations \eqref{gamma-relations}, where in particular
$\gamma \geq \max\{\gamma_0, \gamma_1\}$. Then $g(t)$ is 
$(\A, \overline{\gamma}, k)$-H\"older regular for each fixed $t\in [0,T]$
with $\overline{\gamma}= \min\{\gamma_0, \gamma_1\} \leq \gamma$.
\end{thm}

\begin{proof}
We need to check regularity of the various curvatures in the sense of 
Definition \ref{regular-geometry}. We will only write out the argument for the Riemannian curvature tensor.
The argument for the Ricci curvature tensor is similar.
Recall the following transformation rule for the Riemannian curvature tensor under 
conformal transformations
\begin{align}\label{conformal}
R(e^{2\phi}g) = e^{2\phi} \left( R(g) - \left[g \wedge \left(\nabla \partial \phi - \partial \phi \cdot \partial \phi 
+ \frac{1}{2} \|\nabla \phi\|^2 g\right)
\right]\right), 
\end{align}
where $\wedge$ refers here to the Kulkarni-Nomizu product.
Setting $e^{2\phi}:=(1+u)$, we conclude from  
$u \in \mathcal{C}^{k+2,\A}_{\textup{ie}}(M\times [0,T])_{\gamma_1} $ that the 
components of $R((1+u)g) - (1+u) R(g)$ acting on $x^{-1}\V$ vector fields are in $x^{-2+\gamma_1}\hok$. Now consider the full solution  
$g(t)=(1+u)g + \w$ with the higher order term $\w \in \mathcal{C}^{k+2,\A}_{\textup{ie}}(M\times [0,T], S_0)_{\gamma_0}$. 
Then, $R((1+u)g + \w) - R((1+u)g)$ is an intricate combination of $u$ and $\w$, involving their second order $x^{-2}\V^2$
derivatives and hence its components are in $x^{-2+\min\{\gamma_0, \gamma_1\}}\hok$.
\end{proof}

\section{Small perturbation of flat edge metrics}\label{small}

Let $(M,h)$ be an admissible incomplete edge manifold. Assume that $h$ is flat\footnote{Note 
that a flat edge metric $h$ is automatically H\"older regular with any $(\alpha, k, \gamma)$.}, 
which is equivalent to Ricci flatness in dimension three and is true in case of flat orbifolds. Long time existence
and stability of Ricci flow for small perturbations of Ricci flat metrics that are 
\emph{not} flat, requires an integrability condition and other intricate geometric arguments.
This has been the focus of the joint work with Kr\"oncke \cite{KrVe}. \medskip

In the flat setting we redefine the H\"older spaces in Definition \ref{funny-spaces}
by replacing all edge derivatives $\V$ by $\nabla_{\V}$, where $\nabla$ is the 
covariant derivative on $S$ induced by the Levi Civita connection. We also relax the
condition of tangential stability.

\begin{defn}\label{weak-stability-definition} We say that an admissible edge manifold $(M,h)$
is \emph{weakly} tangentially stable with bound $u$ if 
\begin{equation}\label{stability-flat}
\begin{split}
&\textup{Spec} \, \square_L \geq 0, \quad \textup{Spec} \, \square'_L \geq 0, \\
&u:= \min \left\{\, \textup{Spec} \, \square_L \backslash \{0\}, \textup{Spec} \, \square'_L \backslash \{0\} \, \right\}. 
\end{split}
\end{equation}
\end{defn}

In a joint follow-up work with Kr\"oncke \cite[Theorem 1.7]{KrVe} weak 
tangential stability has been explicitly characterized in terms of the 
spectral data on the cross section as follows. 

\begin{thm}
	Let $(F,g_F)$ be a compact Einstein manifold of dimension $f\geq 3$ 
	with the Einstein constant $(f-1)$. We write $\Delta_E$ for its Einstein operator, and denote the Laplace Beltrami 
	operator by $\Delta$. Then weak tangential stability holds if and only if $\mathrm{Spec}(\Delta_E|_{TT})\geq 0$ and 
	$\mathrm{Spec}(\Delta)\setminus \left\{0\right\}\cap (f, 2(f+1))=\varnothing$. 
\end{thm}

The basic examples of spaces that are weakly tangentially stable but not 
tangentially stable are spaces with cross sections $\mathbb{S}^f$ and $\R \mathbb{P}^f$, 
or quotients of these. We refer to our work \cite{KrVe} for further details. 
\medskip

Under the assumption of weak tangential stability with bound $u$ we define
\begin{equation}\label{minimal-element-flat}
\mu := \sqrt{u + \left(\frac{f-1}{2}\right)^2} - \left(\frac{f-1}{2}\right).
\end{equation}
Note that here we do not treat the pure-trace and the trace-free components 
$S=S_0 \oplus S_1$ separately with different weights.
We also set for any $\gamma > 0$ and a fixed integer $k \in \N_0$
\begin{align*}
H_{\gamma} = \mathcal{C}^{k+2,\A}_{\textup{ie}}(M\times [0,\infty), S)^b_{\gamma}.
\end{align*}

\begin{thm}
Let $(M,h)$ be an admissible flat incomplete edge manifold, which is 
weakly tangentially stable with bound $u$. Consider any $\gamma \in (0,\mu)$,
where $\mu$ is defined by \eqref{minimal-element-flat}. Then for any 
$\alpha \in (0,\mu - \gamma) \cap (0,1)$ the fundamental solution $e^{-t\Delta_L}$ 
admits the following mapping property
\begin{equation}\label{mapping2}
\begin{split}
&e^{-t\Delta_L}: x^{-2+\gamma}\hok(M\times [0,\infty), S)^b  \to 
\mathcal{C}^{k+2,\A}_{\textup{ie}}(M\times [0,\infty), S)^b_{\gamma} = H_\gamma, \\
&e^{-t\Delta_L}: \mathcal{C}^{k+2,\A}_{\textup{ie}}(M, S)^b_{\gamma} \subset H_\gamma \to 
\mathcal{C}^{k+2,\A}_{\textup{ie}}(M\times [0,\infty), S)^b_{\gamma} = H_\gamma,
\end{split}
\end{equation}
where the first operator involves convolution in time, while the second operator
acts without convolution in time.
\end{thm}

\begin{proof}
Since $(M,h)$ is flat, $\square_L$ is the rough Laplacian on $(F,g_F)$ and 
$\ker \square_L$ consists of elements that are parallel along 
$F$ and hence vanish under application of $\nabla_{\partial_z}$. This corresponds
precisely to the scalar case, where $\Delta_L$ reduces to the Laplace Beltrami operator
and $\square_L$ is the Laplace Beltrami operator of $(F,g_F)$. In that case, 
$\ker \square_L$ also consists of constant functions that vanish under the application of $\partial_z$.
Hence the first statement can be obtained
along the lines of the estimates for the scalar Laplace Beltrami operator in 
Theorem \ref{main-est3}. \medskip

For the second statement, note that without convolution in time, a missing $dt$ integration leads to 
two orders less at ff and td in the estimates of Theorems \ref{main-est1} and \ref{main-est3}.
This is however offset by the fact that the heat operator acts on $\mathcal{C}^{k+2,\A}_{\textup{ie}}(M, S)^b_{\gamma}$ instead of the more singular space $x^{-2+\gamma}\hok(M\times [0,\infty), S)^b$.
Thus we may deduce the second statement again as in Theorem \ref{main-est3}. 
\end{proof}

\begin{prop}\label{kernel-discreteness}
Assume that $\Delta_L$ acting 
on $C^\infty_0(M,S)$ is non-negative and denote its Friedrichs 
self-adjoint extension by $\Delta_L$ again. Then $\Delta_L$ is discrete, 
non-negative and 
\begin{equation}
\forall \, k \in \N_0: \ \ker \Delta_L \subset \mathcal{C}^{k+2,\A}_{\textup{ie}}(M, S)^b_{\gamma} \subset H_\gamma.
\end{equation}
\end{prop}

\begin{proof}
By Theorem \ref{essential}, the heat operator $e^{-t\Delta_L}$ coincides with the fundamental
solution constructed in Theorem \ref{heat-asymptotics}. One can easily check from the microlocal
description that the Schwartz kernel of $e^{-t\Delta_L}$ is square-integrable on $M\times M$ for 
fixed $t>0$. Hence $e^{-t\Delta_L}$ is Hilbert Schmidt and due to the semi-group property 
in fact trace-class. Consequently, the Friedrichs extension $\Delta_L$ admits discrete spectrum.
Its non-negativity follows from non-negativity of $\Delta_L$ on $C^\infty_0(M,S)$. \medskip

For fixed $t>0$ we may employ the heat kernel asymptotics to conclude that 
$e^{-t\Delta_L}$ maps $L^2(M,S)$ to $\ho(M,S)$. Since $e^{-t\Delta_L} \restriction 
\ker \Delta_L \equiv \textup{Id}$\footnote{Indeed, $\Delta_L$ is discrete and hence 
the heat operator acts as identity on the kernel of $\Delta_L$.} , we conclude that $\ker \Delta_L \subset \ho(M,S)$
and iteratively, using \eqref{mapping2} and $e^{-t\Delta_L} \restriction 
\ker \Delta_L \equiv \textup{Id}$ find that 
\begin{equation}
\ker \Delta_L \subset \mathcal{C}^{k+2,\A}_{\textup{ie}}(M, S)^b_{\gamma} \subset H_\gamma.
\end{equation}
\end{proof}

\begin{thm}
Let $(M,h)$ be an admissible flat incomplete edge manifold, which is 
weakly tangentially stable with bound $u$. Consider any $\gamma \in (0,\mu)$,
where $\mu$ is defined by \eqref{minimal-element-flat}, and
$\alpha \in (0,\mu - \gamma) \cap (0,1)$. Assume that $\Delta_L$ acting 
on $C^\infty_0(M,S)$ is non-negative and denote its Friedrichs 
extension by $\Delta_L$ again. Consider the orthogonal decomposition
\begin{equation}\label{decomposition-kernel}
\begin{split}
L^2(M,S) &= \ker \Delta_L  \oplus \left( \ker \Delta_L \right)^\perp, \\ 
v &= v_{=} \oplus v_\perp.
\end{split}
\end{equation}
Then for $\lambda_0>0$ being the first non-zero eigenvalue of $\Delta_L$ 
there exists $C>0$ such that 
\begin{equation}\label{mapping-exponential-estimate}
\forall \, v \in \mathcal{C}^{k+2,\A}_{\textup{ie}}(M, S)^b_{\gamma}:
\| e^{-t\Delta_L} v_\perp \|_{k+\alpha, \gamma} \leq C e^{-t\lambda_0} \| v_\perp \|_{k+\alpha, \gamma}.
\end{equation}
\end{thm}

\begin{proof}
The proof is an adaptation of the corresponding argument in the follow-up work 
jointly with Kr\"oncke \cite{KrVe}.
For any $v \in \mathcal{C}^{k+2,\A}_{\textup{ie}}(M, S)^b_{\gamma} \subset L^2(M,S)$, 
we conclude by Proposition \ref{kernel-discreteness}
\begin{equation}
v_=, v_\perp \in \mathcal{C}^{k+2,\A}_{\textup{ie}}(M, S)^b_{\gamma} \subset H_\gamma.
\end{equation}
Hence $e^{-t\Delta_L} v^\perp \equiv e^{-t\Delta^\perp_L} v^\perp \in H_\gamma$ by the mapping properties \eqref{mapping2},
and it makes sense to estimate its norm. 
Denote the set of eigenvalues and eigentensors of the Friedrichs extension $\Delta_L$
by $\{\lambda, v_\lambda\}$. Assume the eigenvalues $\{\lambda\}$ are ordered in the ascending
order and $\lambda_0$ denotes the first non-zero eigenvalue. By discreteness of the spectrum, the heat kernel can be written in terms of 
eigenvalues and eigentensors for any $(p,q) \in M \times M$ by 
\begin{equation}\label{heat-perp}
\begin{split}
&e^{-t\Delta_L}(p,q) = \sum_{\lambda \geq 0} e^{-t\lambda} v_\lambda(p) \otimes v_\lambda(q), \\
&e^{-t\Delta^\perp_L}(p,q) = \sum_{\lambda \geq \lambda_0} e^{-t\lambda} v_\lambda(p) \otimes 
v_\lambda(q).
\end{split}
\end{equation}
Consider any $D\in \{\textup{Id}, \nabla_{\V}\}$. The notation $(D_1 \circ D_2) e^{-t\Delta_L}$
indicates that the operator $D$ is applied once in the first spacial variable of $e^{-t\Delta_L}$ 
and once in the second spacial variable. By Theorem \ref{heat-asymptotics}, the lifted kernel 
$\beta^*(D_1 \circ D_2)  e^{-t\Delta_L}$ is bounded at the left and right face of the heat space 
$\mathscr{M}^2_h$. Consequently, for a fixed $t_0>0$, the pointwise trace 
$\textup{tr}_p (D_1 \circ D_2)  e^{-t_0\Delta_L}(p,p)$ is bounded uniformly in $p\in M$. 
By Proposition \ref{kernel-discreteness}, same holds for $e^{-t_0\Delta^\perp_L}$
and hence there exists $C'(t_0)>0$ such that (we denote the pointwise norm on fibres of $S$ by $\| \cdot \|$)
\begin{equation}
\begin{split}
C'(t_0) &\geq \textup{tr}_p (D_1 \circ D_2) e^{-t_0\Delta^\perp_L}(p,p) = 
\sum_{\lambda \geq \lambda_0} e^{-t\lambda} \| D v_\lambda(p)\|^2 \\
&= e^{-t_0 \lambda_0}\sum_{\lambda \geq \lambda_0} 
e^{-t(\lambda -\lambda_0)} \| D v_\lambda(p)\|^2 
=: e^{-t_0 \lambda_0} \cdot K(t_0,p).
\end{split}
\end{equation}
Note that each $(\lambda-\lambda_0)$ in the sum above is non-negative. Hence
each $e^{-t(\lambda - \lambda_0)}$ as well as $K(t,p)$ are monotonically decreasing as $t \to \infty$ by construction.
Consequently, for any $t\geq t_0$ and any $p\in M$, we conclude
\begin{equation}
K(t,p) \leq C'(t_0) e^{t_0 \lambda_1} =: C(t_0).
\end{equation}
Hence we can estimate for any $t\geq t_0$ and $p\in M$
\begin{equation}
\begin{split}
\textup{tr}_p (D_1 \circ D_2) e^{-t\Delta^\perp_L}(p,p) = e^{-t \lambda_0} \cdot K(t,p)
\leq C(t_0) e^{-t\lambda_0}.
\end{split}
\end{equation}
We conclude with the following intermediate estimate
\begin{equation}\label{pointwise-exp-estimate}
\begin{split}
\| D e^{-t\Delta^\perp_L}(p,q)\| &= 
\sum_{\lambda \geq \lambda_0} e^{-t\lambda} \| D v_\lambda(p)\|  \cdot \|v_\lambda(q)\|
\\ & \leq \sum_{\lambda \geq \lambda_0} \frac{e^{-t\lambda}}{2} \| D v_\lambda(p)\|^2 
+ \sum_{\lambda \geq \lambda_0}  \frac{e^{-t\lambda}}{2} \|v_\lambda(q)\|^2 
\leq C(t_0) e^{-t\lambda_0}.
\end{split}
\end{equation}
From there the statement follows for $t\geq t_0$ for some fixed $t_0 >0$.
By \eqref{mapping2}, the norm of $e^{-t\Delta_L} v_\perp$ is bounded up to a 
constant by the norm of $v_\perp$ uniformly for $t \in [0,t_0]$. Hence the 
statement follows for all $t>0$ after a change of constants. 
\end{proof}

\begin{defn}
Let $\varepsilon > 0$. An incomplete edge metric $g$ on $M$ is said to be 
an $\varepsilon$-close higher order perturbation of $h$ in $H_{\gamma}$, if 
$(g-h) \in H_\gamma$ with the H\"older norm smaller than or equal to $\varepsilon$.
\end{defn}

Note that such a higher order perturbation $g$ of an admissible edge metric $h$ 
is automatically admissible as well, by the argument in \S \ref{edge-flow}. \medskip

We study Ricci flow of $g$, and in slight difference to \S \ref{short-section}
apply the Ricci de Turck trick with $h$ as the background metric. This leads to 
the linearized parabolic equation as in \eqref{ricci-de-turck-linearization} with $\textup{scal}(h)$
and $\textup{Ric}'(h)$ being trivially zero for the Ricci flat metric $h$, and $T_3(v)=0$ since $h$ 
is actually assumed to be flat. Writing $v= u \oplus \w$, $\Delta_L$ and $\nabla$ for the Lichnerowicz Laplacian
and the Levi Civita covariant derivative on $S$, defined with respect to $h$, we obtain 
\begin{equation}\label{ricci-de-turck-linearization2}
\begin{split}
\left( \partial_t  + \Delta_L \right) v = -\frac{u}{1+u} 
\Delta_L v + x^{-2} O_1(\nabla_{\V} v)
O_1(v, \nabla_{\V} v)  =: F(v).
\end{split}
\end{equation}
We seek to find a solution $g(t)=(1+u)h\oplus \w$ to that equation with initial condition $g(0)=g$.
Here, as before $(1+u)h\oplus \w$ denotes the decomposition into pure trace and 
trace-free components with respect to $h$. We prove the following theorem. 

\begin{thm}\label{small-thm}
Consider an admissible flat edge manifold $(M,h)$ with an edge singularity at $B$,
satisfying weak tangential stability with bound $u$. Assume that $\Delta_L$ acting 
on $C^\infty_0(M,S)$ is non-negative and denote its Friedrichs self-adjoint extension 
by $\Delta_L$ again. Then there exists $\varepsilon > 0$ sufficiently 
small such that if $g$ is an $\varepsilon$-close higher 
order perturbation of $h$, with 
\begin{equation}
(g-h) \perp \ker \Delta_L \subset L^2(M,S),
\end{equation}
the Riemannian metric $g$ may be evolved under the 
Ricci de Turck flow as $g(t) = (1+u) \oplus \w$
within the Banach space $H_\gamma$ for all times,
provided the following algebraic relations are satisfied
\begin{equation}\label{gamma-relations-2}
\begin{split}
&\gamma \in (0, \mu), \ \textup{and if $\dim B > 0$ then} \, 
\gamma \leq 2,
\end{split}
\end{equation}
and $\A \in (0,(\mu - \gamma))$.
Moreover there exists $\mu(\varepsilon)>0$
sufficiently small, with $\mu(\varepsilon)\to 0$ as $\varepsilon$ goes to zero, such that 
the H\"older norm of $(g(t)-h)$ in $H_{\gamma}$ is smaller or equal to $\mu(\varepsilon)$,
uniformly in time $t\in [0,\infty)$.
\end{thm}

\begin{proof}
The Ricci de Turck flow $g(t)$ with $h$ as background metric and 
$g(0)=g$ as initial condition exists is a fixed point of the following map
\begin{align*}
\Psi : v := g(t)-h \in H_{\gamma} \mapsto e^{-t\Delta_L} * F(v) + 
e^{-t\Delta_L} (g-h) \in H_{\gamma},
\end{align*}
where $\Delta_L$ is the Friedrichs self adjoint extension of the Lichnerowicz Laplacian on $S$, 
$e^{-t\Delta_L}$ is the corresponding heat operator, $*$ refers to the action of the heat operator 
with convolution in time, and in $e^{-t\Delta_L}(g-h)$ the heat operator is
applied without convolution in time. The fact that $\Psi$ maps $H_\gamma$
to itself follows from 
\begin{equation}
(g-h) \in H_\gamma, \quad F(v)\in x^{-2+\gamma}\hok(M\times [0,\infty), S)^b,
\end{equation} 
for $v \in H_\gamma$, and the mapping properties \eqref{mapping2}.\medskip

Consider the orthogonal decomposition \eqref{decomposition-kernel}.
We fix any $\beta \in (0,\lambda_0)$. In order to prove existence of a 
fixed point for $\Psi$, we consider any 
$\delta > 0$ and restrict $\Psi$ to a subset
of the Banach space $H_\gamma$
\begin{align*}
Z_\delta := \{v \in H_{\gamma} \mid \|v_\perp(t) \|_{k+\alpha, \gamma} \leq \delta
e^{-t(\lambda_0-\beta)},  \ \|v_=(t) \|_{k+\alpha, \gamma} \leq \delta\}.
\end{align*}
Note that on flat manifolds $\Delta_L = \nabla^* \nabla$ and hence
\begin{equation}
\ker \Delta_L = \ker \nabla \subset L^2(M,S).
\end{equation}
Note that $x^{-2+\gamma}\hok(M\times [0,\infty), S)^b \subset L^2(M,S)$, since we always assume 
$\dim F \geq 1$. Since in $F(v)$ all terms are quadratic in $v$ and admit at least one component
of the form $\nabla v = \nabla v_{\perp}$, we conclude for any $v \in Z_\delta$ and some uniform constants $C,C'>0$
\begin{equation}\label{F-estimate}
\| F(v(t)) \|_{L^2} \leq C \| F(v(t)) \|_{k+\alpha, -2+\gamma} \leq CC' \delta^2 e^{-t(\lambda_0-\beta)}.
\end{equation} 
Consider the discrete set $\{\lambda, v_\lambda\}$ 
of eigenvalues and eigentensors of the Friedrichs extension $\Delta_L$.
As in \eqref{heat-perp}, we may now decompose the heat kernel for any $(p,q) \in M \times M$ as follows
\begin{equation*}
\begin{split}
e^{-t\Delta_L}(p,q) &= \sum_{\lambda = 0} v_\lambda(p) \otimes v_\lambda(q) + 
\sum_{\lambda \geq \lambda_0} e^{-t\lambda} v_\lambda(p) \otimes 
v_\lambda(q) \\ &=: \Pi(p,q) + e^{-t\Delta^\perp_L}(p,q).
\end{split}
\end{equation*}
Clearly, $\Pi$ is the orthogonal projection of $L^2(M,S)$ onto $\ker \Delta_L$, while
$e^{-t\Delta^\perp_L}$ is the composition of the heat operator with the orthogonal
projection onto $(\ker \Delta_L)^\perp$. Hence we find for any $v \in H_\gamma$
(recall, $F(v) \in L^2(M,S)$)
\begin{equation*}
\begin{split}
\left(e^{-t\Delta_L}* F(v)\right)_= = \Pi * F(v),  \quad
\left(e^{-t\Delta_L}* F(v)\right)_\perp = e^{-t\Delta^\perp_L} * F(v).
\end{split}
\end{equation*}
In view of \eqref{F-estimate} we may estimate the action of $\Pi$ for any $v \in Z_\delta$
as follows
\begin{equation*}
\begin{split}
\| \left(e^{-t\Delta_L}* F(v)\right)_=\|_{k+\alpha, \gamma} = 
&\leq \sum_{\lambda = 0} \| v_\lambda \|_{k+\alpha, \gamma}  \int_0^t |(v_\lambda, F(v(t')))_{L^2}| dt'
\\ &\leq \sum_{\lambda = 0} \| v_\lambda \|_{k+\alpha, \gamma} \| v_\lambda \|_{L^2} 
\int_0^t \| F(v(t')) \|_{L^2} dt' \\ 
&\leq C \int_0^t \delta^2 e^{-t'(\lambda_0-\beta)} dt' \leq C' \delta^2,
\end{split}
\end{equation*}
where $C,C'>0$ are some uniform constants. In order to obtain a similar estimate for the action 
of $e^{-t\Delta^\perp_L}$, note that by the pointwise estimate \eqref{pointwise-exp-estimate},
the Schwartz kernel of $e^{-t\Delta^\perp_L}$ can be written as $e^{-t\lambda_0}$
times a kernel $G$ of same asymptotics in the heat space $\mathscr{M}^2_h$, 
which is uniform as $t\to \infty$. Hence we may write $\left(e^{-t\Delta_L}* F(v)\right)_\perp$ 
as follows
\begin{equation*}
\begin{split}
\int_0^t \int_M e^{-(t-t') \lambda_0} e^{-t'(\lambda_0-\beta)} 
G(t-t',p,q) \left(e^{t'(\lambda_0-\beta)} F(v)(t',q)\right) dt' d\textup{vol}_h(q).
\end{split}
\end{equation*}
We now estimate for any $v \in Z_\delta$ the $H_\gamma$-norm and find
using \eqref{F-estimate}
\begin{equation*}
\begin{split}
\| \left(e^{-t\Delta_L}* F(v)\right)_\perp\|_{k+\alpha, \gamma}
&\leq e^{-t(\lambda_0-\beta)} \| \int_0^t \int_M 
G(t-t',p,q) \times \\ & \times \left(e^{t'(\lambda_0-\beta)} F(v)(t',q)\right) dt' d\textup{vol}_h(q)\|_{k+\alpha, \gamma} 
\\ &\leq C \delta^2 e^{-t(\lambda_0-\beta)}
\end{split}
\end{equation*}
for some uniform constant $C>0$. Note also that by assumption, $(g-h) \perp \ker \Delta_L$
with $H_\gamma$-norm bounded by $\varepsilon$. Hence, by Theorem 
\ref{mapping-exponential-estimate} 
\begin{equation*}
\begin{split}
\| \left(e^{-t\Delta_L} (g-h) \right)_\perp\|_{k+\alpha, \gamma}
= \| e^{-t\Delta^\perp_L} (g-h) \|_{k+\alpha, \gamma}
\leq C \varepsilon e^{-t\lambda_0}.
\end{split}
\end{equation*}
Summarizing we have shown that there exists a uniform constant $C>0$
such that for any $v \in Z_\delta$
\begin{equation}
\| \left(\Psi v \right)_\perp\|_{k+\alpha, \gamma} \leq C \left(\delta^2 + \varepsilon \right)
e^{-t(\lambda_0-\beta)}, \quad \| \left(\Psi v \right)_= \|_{k+\alpha, \gamma} \leq C \delta^2.
\end{equation} 
Taking $(\varepsilon, \delta)$ sufficiently small (proportionally to each other) ensures
that $\Psi$ maps $Z_\delta$ to itself. Moreover, since $F(v)$ is quadratic in $v$, we find
that $\Psi$ is a contraction
\begin{align*}
\|\Psi(v) - \Psi(v')\|_{H_{\gamma}} \leq q \|v-v'\|_{H_{\gamma}}
\end{align*}
with some positive $q<1$, for all $v,v' \in Z_\delta$.
Hence a fixed point exists in $Z_\delta \subset H_{\gamma}$. Note that $\delta>0$ can be taken 
smaller the smaller we choose $\varepsilon>0$.\medskip

Note that in contrast to Theorem \ref{existence1}, we do not need to 
restrict to a finite time interval $[0,T]$ with $T>0$ sufficiently small
and set up the fixed point argument in the H\"older space $H_{\gamma}$ for all times.
This is due to the fact that all terms in the linearization of the Ricci de Turck 
flow \eqref{ricci-de-turck-linearization2} are at least quadratic and hence 
$\Psi$ maps $Z_\delta$ to itself for $\delta>0$ sufficiently small without additional
restrictions on time. 
\end{proof}

Note that as explained in \S \ref{edge-flow}, the Ricci de Turck flow $g(t)$
is an admissible edge metric with the same leading term as $h$ up to a 
conformal transformation of the metric along the
edge singularity and a change of the boundary defining function $x$. 

\section{Appendix: Mean value theorem on edge manifolds}\label{mean-section}

The subsequent section on mapping properties of the heat kernel
for the Lichnerowicz Laplacian requires an estimate of the corresponding
H\"older differences. This will be somewhat different from similar estimates 
performed in \cite{BV}, since the Lichnerowicz Laplacian on symmetric $2$-tensors
does not satisfy stochastic completeness. Therefore we will use some different 
argument, which is developed in the present section. We begin with the following 
consequence of the mean value theorem for Banach-valued functions of a single variable.

\begin{lemma}\label{mean-lemma}
Consider any Banach space $B$ with norm $\|\cdot \|$ and any $\eta, \eta' \in \R$
contained in a compact convex subset $K\subset \R$. Assume $\eta \leq \eta'$. Consider some continuously 
differentiable function $\w: K \to B$. Then for  and any fixed odd integer $N\in \N$, 
there exists a uniform constant $C>0$ and some $\delta \in [\eta, \eta']$ such that 
\begin{align}
\| \w (\eta) - \w (\eta') \| \leq C |\eta-\eta'|^{\frac{1}{N}} 
 \left\| \delta^{\frac{N-1}{N}} \left( \frac{d}{d\eta} \w \right) (\delta) \right\|
\end{align}
\end{lemma}

\begin{proof}
Define $o:= \eta^{\frac{1}{N}} \in \R$ and $o':= \eta'^{\frac{1}{N}} \in \R$ for any $\eta, \eta' \in \R$.
By the mean value theorem in Banach spaces, there exists some
$\delta \in [\eta, \eta']$ (we write $\xi:= \delta^{\frac{1}{N}}$)
\begin{equation}\label{o}
\begin{split}
\| \w(\eta) - \w(\eta')\| &= \| \w(o^N) - \w(o'^N) \| \leq |o-o'|  \left\| \left(\frac{d}{do} \w \right) (\xi) \right\|
\\ &= N |\eta^{\frac{1}{N}} - \eta'^{\frac{1}{N}}| \,  \left\| \delta^{\frac{N-1}{N}} 
\left( \frac{d}{d\eta} \w \right) (\delta) \right\|.
\end{split}
\end{equation}
One computes using l'Hospital for any $N>1$
\begin{align}
\lim_{\eta \to \eta'} \frac{(\eta^{\frac{1}{N}} - \eta'^{\frac{1}{N}})}
{(\eta-\eta')^{\frac{1}{N}}} = \lim_{q\to 1} \frac{q^{\frac{1}{N}}-1}{(q-1)^{\frac{1}{N}}}
=  \lim_{q\to 1} \frac{(q-1)^{1-\frac{1}{N}}}{q^{1-\frac{1}{N}}} = 0.
\end{align}
Consequently, for any $\eta, \eta' \in \R$ contained in a compact convex subset $K\subset \R$ 
and any $N>1$ there exists a uniform constant $C=C(N,K)>0$ such that 
\begin{align}\label{C}
N |\eta^{\frac{1}{N}} - \eta'^{\frac{1}{N}}| \leq C 
|\eta-\eta'|^{\frac{1}{N}}.
\end{align}
Taking Hilbert space norm on both sides of \eqref{o} proves the statement of the lemma
using the estimate \eqref{C}.
\end{proof}

As a consequence of the previous lemma we conclude with the following corollary,
where for simplicity we assume that the edge $B$ as well as the fibre $F$ are one-dimensional. The general 
case is discussed verbatim.

\begin{cor}\label{mean-cor}
Consider any continuously differentiable section 
$\w \in \Gamma (\textup{Sym}^2_0({}^{ie}T^*M))$. Consider two copies 
of local coordinates $(x,y,z), (x',y',z') \in \U$. Consider 
any $\wy \in B$ lying in the same (convex) coordinate chart as $y$ and $y'$. 
Then there exist $\xi\in (0,1)$ lying in the line segment between $x$ and $x'$; 
$\gamma$ lying in the line segment between $y$ and $y'$;
and $\zeta \in F$ lying in the line segment connecting $z, z'\in F$; 
as well as a constant $C>0$ depending only on the choice of local coordinate charts
and the odd integer $N \in \N$, such that
\begin{equation*}
\begin{split}
\frac{\| \w (x,y-\wy,z) - \w (x',y'-\wy,z') \|}{d_M((x,y,z), (x',y',z'))^{\frac{1}{N}} }  
\leq & \, C \left( \| \xi^{\frac{N-1}{N}} \partial_\xi \w (\xi, y, z) \| \right.
\\ & \left. + \, \| \gamma - \wy \|^{\frac{N-1}{N}} \| \partial_\gamma \w (x', \gamma-\wy, z) \| \right.
\\ & \left. + \, \| x'^{-\frac{1}{N}}\partial_\zeta \w (x', y'-\wy, \zeta)\| \right).
\end{split}
\end{equation*}
\end{cor}

\begin{proof}
We write the difference $\w (x,y-\wy,z) - \w (x',y'-\wy,z')$ as follows
\begin{equation*}
\begin{split}
\w (x,y-\wy,z) - \w (x',y'-\wy,z') &= \w (x,y-\wy,z) - \w (x',y-\wy,z)
\\ &+ \w \left(x', y -\wy, z\right) - 
\w \left(x', y' -\wy, z\right) \\ &+ \w (x',y'-\wy,z) - \w (x',y'-\wy,z').
\end{split}
\end{equation*}
As a direct application of Lemma \ref{mean-lemma} we obtain
\begin{equation}\label{mean-thm}
\begin{split}
\| \w (x,y-\wy,z) - \w (x',y'-\wy,z') \| &\leq C\, |x-x'|^{\frac{1}{N}}  \| \xi^{\frac{N-1}{N}} \partial_\xi \w (\xi, y, z) \|
\\ &+ C\, \| y-y'\|^{\frac{1}{N}} \| \gamma - \wy \|^{\frac{N-1}{N}} \| \partial_\gamma \w (x', \gamma-\wy, z) \|
\\ & + C\, \| z-z'\| \, \| \partial_\zeta \w (x', y'-\wy, \zeta)\|.
\end{split}
\end{equation}
for some uniform constant $C>0$. From here the statement of the theorem follows.
\end{proof}

\section{Appendix: Comparison of various H\"older spaces}\label{spaces-comparison}

H\"older spaces on spaces with incomplete edge singularities
have been an important tool in studying K\"ahler-Einstein edge metrics in \cite{JMR},
as well as in the discussion of the Yamabe flow in \cite{BV}. \medskip

We shall provide a brief overview how the spaces here and in \cite{JMR,BV} are related.  
Let us start with the definition of the wedge H\"older spaces 
(of time-independent scalar functions) as in \cite{JMR}.
\begin{defn}\label{JMR1}
The wedge H\"older space $\mathcal{C}^{0,\alpha}_\omega (M), \A\in (0,1),$ consists of functions 
$u(p)$ that are continuous on $\overline{M}$ with finite $\A$-th H\"older 
norm
\begin{align}
\|u\|_{\A}:=\|u\|_{\infty} + \sup \left(\frac{|u(p)-u(p')|}{d_M(p,p')^{\A}}\right) <\infty, 
\end{align}
where the distance function $d_M(p,p')$ between any two points $p,p'\in M$ 
is defined with respect to the incomplete edge metric $g$. The higher order 
wedge H\"older spaces are defined for any order $k\in \N$ by
\begin{equation} 
\mathcal{C}^{k,\alpha}_\omega (M) := \{ u \in \mathcal{C}^{0,\alpha}_\omega (M) \cap C^k(M)
\mid (x^{-1}\V)^j u \in \mathcal{C}^{0,\alpha}_\omega (M) \ \textup{for any} \ j\leq k\}.
\end{equation}
The weighted wedge H\"older spaces are defined for any weight $\gamma \in \R$ by
\begin{equation} 
x^\gamma \mathcal{C}^{k,\alpha}_\omega (M) := \{ x^\gamma u
\mid u \in \mathcal{C}^{k,\alpha}_\omega (M)\}.
\end{equation}
\end{defn} 

\cite{JMR} also introduces the edge H\"older spaces of time-independent scalar functions,
which are defined with respect to a complete edge metric and different derivatives. More precisely, we have the following.
\begin{defn}\label{JMR2}
The edge H\"older space $\mathcal{C}^{0,\alpha}_e (M), \A\in (0,1),$ consists of functions 
$u(p)$ that are continuous on $\overline{M}$ with finite $\A$-th H\"older 
norm
\begin{align}
\|u\|_{\A}:=\|u\|_{\infty} + \sup \left(\frac{|u(p)-u(p')|}{D_M(p,p')^{\A}}\right) <\infty, 
\end{align}
where the distance function $D_M(p,p')$ between any two points $p,p'\in M$ 
is defined with respect to the complete edge metric $x^{-2}g$. The higher order 
edge H\"older spaces are defined for any order $k\in \N$ by
\begin{equation} 
\mathcal{C}^{k,\alpha}_e (M) := \{ u \in \mathcal{C}^{0,\alpha}_e (M) \cap C^k(M)
\mid \V^j u \in \mathcal{C}^{0,\alpha}_e (M) \ \textup{for any} \ j\leq k\}.
\end{equation}
The weighted edge H\"older spaces are defined for any weight $\gamma \in \R$ by
\begin{equation} 
x^\gamma \mathcal{C}^{k,\alpha}_e (M) := \{ x^\gamma u
\mid u \in \mathcal{C}^{k,\alpha}_e (M)\}.
\end{equation}
\end{defn} 

As explained in \cite[\S 2.6.3]{JMR}, the wedge and the edge H\"older spaces of Definitions 
\ref{JMR1} and \ref{JMR2} are related by $\mathcal{C}^{k,\alpha}_\omega (M) \subsetneq 
\mathcal{C}^{k,\alpha}_e (M)$. The present work also features H\"older spaces of time-independent scalar functions
$\hok (M)_\gamma$, that are closely related to the wedge H\"older spaces of Definition \ref{JMR1}.
Namely, we clearly have $\ho(M)_\gamma \equiv x^\gamma \mathcal{C}^{0,\alpha}_\omega (M)$. That
equality does not extend to higher order spaces. however. In fact for any $u \in \hok(M)_0$ we find 
$(x^{-1}\V)^j u \in x^{-j}\mathcal{C}^{0,\alpha}_\omega (M)$. Thus
\begin{equation}
\hok(M)_\gamma = \bigcap_{j=0}^k \, x^{\gamma-j} \mathcal{C}^{j,\alpha}_\omega (M).
\end{equation}
In the work \cite{BV}, we have introduced another type of H\"older spaces, built upon
$\ho(M\times [0,T])$ by requiring H\"older regularity under iterative differentiation by a subset of $x^{-1} \V$.
More precisely, \cite{BV} defines the following.

\begin{defn}\label{BV-def}
Let $\Delta$ denote the Laplace-Beltrami operator of $(M,g)$.
\begin{equation*}
\begin{split}
&\mathcal{C}^{1+\A}_{\textup{ie}} (M\times [0,T]) = 
\{u\in \ho(M\times [0,T]) \mid x^{-1}\mathcal{V}_e u \in \ho(M\times [0,T])\}, \\
&\mathcal{C}^{2+\A}_{\textup{ie}} (M\times [0,T]) = \{u\in \mathcal{C}^{\A}_{\textup{ie}} (M\times [0,T]) 
\mid \Delta u, x^{-1}\mathcal{V}_e u, \partial_t u \in \mathcal{C}^{\A}_{\textup{ie}}(M\times [0,T]) \}, \\ 
&\mathcal{C}^{2k + 1+\A}_{\textup{ie}} (M\times [0,T]) = \{u\in \mathcal{C}^{1+\A}_{\textup{ie}} (M\times [0,T]) \mid \Delta^j u \in 
\mathcal{C}^{1+\A}_{\textup{ie}} (M\times [0,T]) , j\leq k \}, \\
&\mathcal{C}^{2(k+1)+\A}_{\textup{ie}} (M\times [0,T]) = \{u\in \mathcal{C}^{2+\A}_{\textup{ie}} (M\times [0,T]) \mid \Delta^j u \in \mathcal{C}^{2+\A}_{\textup{ie}} (M\times [0,T]) , j\leq k\}.
\end{split}
\end{equation*}
\end{defn}

The restricted set of derivatives that appear in the Definition \ref{BV-def} simplified various
heat kernel estimates in \cite{BV}. The relation to the H\"older spaces above is as follows. 
Extending Definition \ref{JMR1} to include H\"older regularity in time, we find
\begin{equation}\begin{split}
&\mathcal{C}^{1+\A}_{\textup{ie}} (M\times [0,T]) \equiv \mathcal{C}^{1,\alpha}_\omega (M\times [0,T]), \\
&\mathcal{C}^{k+\A}_{\textup{ie}} (M\times [0,T]) \subsetneq \mathcal{C}^{k,\alpha}_\omega (M\times [0,T]), \ k\geq 2.
\end{split}\end{equation}

\def\cprime{$'$}
\providecommand{\bysame}{\leavevmode\hbox to3em{\hrulefill}\thinspace}
\providecommand{\MR}{\relax\ifhmode\unskip\space\fi MR }
\providecommand{\MRhref}[2]{%
  \href{http://www.ams.org/mathscinet-getitem?mr=#1}{#2}
}
\providecommand{\href}[2]{#2}

\end{document}